\documentclass[11pt,a4paper,reqno]{amsart}

\usepackage{color}
\usepackage{amsmath}
\usepackage{amssymb}
\usepackage{amsfonts}
\usepackage{latexsym}

\newtheorem{theorem}{Theorem}[section]
\newtheorem{proposition}[theorem]{Proposition}
\newtheorem{lemma}[theorem]{Lemma}
\newtheorem{corollary}[theorem]{Corollary}
\newtheorem{definition}[theorem]{Definition}
\newtheorem{example}[theorem]{Example}

\newtheorem{remark}[theorem]{Remark}
\newtheorem{assumption}[theorem]{Assumption}

\newcommand{\wto}{\rightharpoonup}
\newcommand{\wsto}{\overset{*}{\wto}}
\newcommand{\toto}{\rightrightarrows}

\newcommand{\real}{{\mathbb{R}}}

\newcommand{\nat}{{\mathbb{N}}}

\newcommand{\veps}{\varepsilon}

\newcommand{\osc}{\operatornamewithlimits{osc}}
\newcommand{\playr}{\mathcal{P}_r}
\newcommand{\stopr}{\mathcal{S}_r}

\begin{document}

\title[Newton and Bouligand Derivatives of the Play and the Stop]{Newton and
Bouligand Derivatives of the Scalar Play and Stop Operator}

\author[M. Brokate]{Martin Brokate}

\address{Martin Brokate \hfill\break
Dept. of Mathematics, Technical University of Munich, Boltzmannstr. 3,
D-85747 Garching, Germany}
\email{brokate@ma.tum.de}


\begin{abstract}
We prove that the play and the stop operator possess Newton and Bouligand
derivatives, and exhibit formulas for those derivatives. 
The remainder estimate is given in a strengthened form, and a
corresponding chain rule is developed.
The construction of the Newton derivative ensures that the mappings involved 
are measurable.
\end{abstract}

\maketitle

\section{Introduction.}
The aim of this paper is to show that the play and the stop operator possess
Newton as well as Bouligand derivatives, and to compute those derivatives.
Newton derivatives are needed when one wants to solve equations
\[
F(u) = 0
\]
for nonsmooth operators $F$ by Newton's method with a better than linear
convergence rate. Bouligand derivatives are closely related to Newton derivatives,
and can be used to provide sensitivity results as well as optimality conditions
for problems involving nonsmooth operators.

The scalar play operator  and its twin, the scalar stop operator, act on
functions $u:[a,b]\to\real$ and yield functions $w = \playr[u;z_0]$ and
$z = \stopr[u;z_0]$ from $[a,b]$ to $\real$. The number $z_0$ plays the
role of an initial condition. Their formal definition, in the spirit of
\cite{KDE70}, is given below in Section 6; alternatively, they arise as solution 
operators of the evolution variational inequality
\begin{subequations}\label{i.1}
\begin{gather}
\dot{w}(t)\cdot(\zeta - z(t)) \le 0 \,,
\quad \text{for all $\zeta\in [-r,r]$,}
\\
z(t) \in [-r,r] \,, \quad z(a) = z_0\in [-r,r] \,,
\\
w(t) + z(t) = u(t) \,. \label{i.1.3}
\end{gather}
\end{subequations}
The play and the stop operator are rate-independent; in fact, they
constitute the simplest nontrivial examples of rate-independent operators
\cite{Vis,BS,K,MR} if one disregards relays whose nature is inherently
discontinuous. Due to (\ref{i.1.3}),
their mathematical properties are closely related.

A lot is known about the play and the stop. Viewed as operators between
function spaces, their typical regularity is Lipschitz (or less).
In particular, they are not differentiable in the classical sense. The question
whether weaker derivatives (e.g., directional derivatives) exist was
addressed, to the author's knowledge, for the first time in \cite{BK}
where it was shown that the play and the stop are directionally
differentiable from $C[a,b]$ to $L^p(a,b)$ for $p < \infty$.
(This is not to be confused with the existence and form of
time derivatives of functions like $t\mapsto \playr[u;z_0](t)$,
for which there are many results available.)

The results below serve to narrow the gap between differentiability
and non-differentiability of rate-independent operators. Their proofs
given here are
based on the same idea as used in \cite{BK}, namely, to locally
represent the play as a composition of operators whose main
ingredient is the cumulated maximum.

It is natural to ask whether it is possible to prove weak differentiability
of the play and the stop operator in the framework of the variational 
formulation (\ref{i.1}). Indeed, for elliptic variational inequalities, 
a large body of literature is available, going back to \cite{Mig}. 
In that case, the solution operator is closely linked to the metric
projection onto convex sets whose differentiability properties also
have been analyzed for a long time. 
For evolution variational inequalities of parabolic type,
we refer to the recent contribution \cite{Chr} and the literature cited
there. For rate independent variational inequalities, corresponding
results do not seem to exist, not even for the ODE case given in (\ref{i.1}).

Our main results are given in Theorem \ref{po.gnd} for Newton
differentiability and Theorem \ref{bd.p} for Bouligand differentiability
of the play. They are based on corresponding results for the
maximum functional (Proposition \ref{nmf.nbder}) and the cumulated
maximum operator (Proposition \ref{am.newder}). The extension to
the parametric play is given in Proposition \ref{pp.pnd}.

When attempting to prove Newton differentiability of the play, some issues
arise which complicate matters and are, at least in part, responsible
for the length of this paper.
First, the construction of the Newton derivative of the play leads
to a set-valued derivative in a natural manner. Its elements $L$ 
should have the property that the first order approximations
$\delta w = L \delta u$ are measurable functions. 
Since Newton derivatives are not obtained as limits, 
and we are dealing with operators between function spaces,
measurability becomes an issue.
Second, with regard to the form of the remainder, we aim
at a somewhat stronger result than standard Newton 
differentiability,
having in mind applications to partial differential equations.
Third, we want to treat not only a single play operator, but 
also a parametric family of play operators, having in mind
problems where play operators e.g. are distributed continuously
over space. Again, the problem of measurability has to be solved.

The proofs of Newton and of Bouligand differentiability are rather
similar; for Bouligand derivatives, some of the problems mentioned
above do not even arise. Nevertheless, we have chosen to elaborate 
the proofs for both cases to some extent; the details are 
somewhat cumbersome and should not be placed too much as a burden 
on the reader.

\section{Notions of derivatives.}
We collect some established notions of derivatives for mappings
\[
F: U\to Y \,,\quad U\subset X \,,
\]
where $X$ and $Y$ are normed spaces, and $U$ is an open subset of $X$.
These notions are classical, but the terminology is not uniform in
the literature.
\begin{definition}\label{nwd.wdef}
(i) The limit, if it exists,
\begin{equation}\label{nwd.wdef.1}
F'(u;h) := \lim_{\lambda\downarrow 0} \frac{F(u+\lambda h) - F(u)}{\lambda} \,,\quad
u\in U \,,\,h\in X \,,
\end{equation}
is called the \textbf{directional derivative} of $F$ at $u$ in the direction $h$.
It is an element of $Y$.
\hfill\\
(ii) If the directional derivative satisfies
\begin{equation}\label{nwd.wdef.2}
F'(u;h) = \lim_{\lambda\downarrow 0} \frac{F(u+\lambda h + r(\lambda)) - F(u)}{\lambda}
\end{equation}
for all functions $r:(0,\lambda_0)\to X$ with $r(\lambda)/\lambda \to 0$
as $\lambda \to 0$, it is called the \textbf{Hadamard derivative}
of $F$ at $u$ in the direction $h$.
\hfill\\
(iii) If the directional derivative exists for all $h\in X$ and satisfies
\begin{equation}\label{nwd.wdef.3}
\lim_{h\to 0} \frac{\|F(u+ h) - F(u) - F'(u;h)\|}{\|h\|} = 0 \,,
\end{equation}
it is called the \textbf{Bouligand derivative} of $F$ at $u$ in the direction $h$.
\hfill\\
(iv) If the Bouligand derivative has the form $F'(u;h) = Lh$ for some linear
continuous mapping $L:X\to Y$, then $L$ is called the
\textbf{Fr\'echet derivative} of $F$ at $u$ and denoted as $DF(u)$.
\hfill\\
(v) The mapping $F$ is called directionally (Hadamard, Bouligand, Fr\'echet, resp.)
differentiable at $u$ (in $U$, resp.), if the corresponding derivative exists
at $u$ (for all $u\in U$, resp.) for all directions $h\in X$.
\hfill$\Box$
\end{definition}
In the definition above, it is tacitly understood that the limits are taken
in the sense ``not equal 0''.

We have $F'(u;\lambda h) = \lambda F'(u;h)$ if $\lambda\ge 0$. This as
well as the following well-known facts are elementary consequences of the
above definitions.
\begin{proposition}\label{nwd.dirhad}
Let $F$ be directionally differentiable and locally Lipschitz continuous at
$u\in U$. Then $F$ is Hadamard differentiable at $u$. Moreover, if $\ell_u$
is a local Lipschitz constant for $F$ at $u$,
\begin{equation}\label{nwd.dirhad.1}
\|F'(u;h_1) - F'(u;h_2)\| \le \ell_u \|h_1 - h_2\| \qquad \forall \; h_1,h_2\in X\,.
\end{equation}
Consequently, if $\ell$ is a global Lipschitz constant for $F$,
\begin{equation}\label{nwd.dirhad.2}
\|F(u+h) - F(u) - F'(u;h)\| \le 2\ell \|h\| \qquad \forall \; h\in X\,.
\end{equation}
\hfill$\Box$
\end{proposition}

\begin{corollary}\label{nwd.dhb}
If $F$ is locally Lipschitz, then directional and Hadamard differentiability
at $u\in U$ are equivalent, and are implied by Bouligand differentiability
at $u$.
\hfill$\Box$
\end{corollary}

In terms of a remainder function, the definition (\ref{nwd.wdef.3}) of Bouligand
differentiability at $u$ is equivalent to
\begin{equation}\label{nwd.rem.1}
\|F(u+h) - F(u) - F'(u;h)\| \le \rho_u(\|h\|)\cdot\|h\| \,,
\end{equation}
where $\rho_u(\delta) \downarrow 0$ for $\delta\downarrow 0$. 
In view of (\ref{nwd.dirhad.2}), we may assume that $\rho_u$ is globally bounded,
\begin{equation}\label{nwd.rem.2}
\rho_u \le 2\ell \qquad \forall \; u\in U\,,
\end{equation}
if $\ell$ is a global Lipschitz constant for $F$.

The notion of a Newton derivative is more recent.
A mapping $G: U \to \mathcal{L}(X,Y)$, the space of all linear and continuous
mappings from $X$ to $Y$, is called a Newton derivative of $F$ in $U$,
if
\begin{equation}\label{nwd.ndef.0}
\lim_{h\to 0} \frac{\|F(u+h) - F(u) - G(u+h)h\|}{\|h\|} = 0
\end{equation}
holds for all $u\in U$.
It is never unique; for example, modifying $G$ at a single point
does not affect the validity of (\ref{nwd.ndef.0}) in $U$.

It has turned out to be natural to allow Newton derivatives to be set-valued.
For set-valued mappings we write ``$f:X\toto Y$''
instead of ``$f:X\to \mathcal{P}(Y)\setminus\emptyset$''.

\begin{definition}\label{nwd.ndef}
A mapping $G: U \toto \mathcal{L}(X,Y)$ is called a \textbf{Newton derivative}
of $F$ in $U$, if
\begin{equation}\label{nwd.ndef.1}
\lim_{h\to 0} \sup_{L\in G(u+h)} \frac{\|F(u+h) - F(u) - L h\|}{\|h\|} = 0
\end{equation}
holds for all $u\in U$. $G$ is called \textbf{locally bounded} if
for every $u\in U$ the sets $\{\|L\|: L\in G(v),\|v-u\|\le \delta\}$ are
bounded for some suitable $\delta = \delta(u)$.
$G$ is called \textbf{globally bounded} if these bounds can be chosen
independently from $u$.
\end{definition}
It is well known that if $F$ is continuously Fr\'echet differentiable in $U$,
then $G(u) = \{DF(u)\}$ is a single-valued Newton derivative of $F$ in $U$.

We write (\ref{nwd.ndef.1}) in remainder form,
\begin{equation}\label{nwd.ndef.4}
\sup_{L\in G(u+h)} \|F(u+h) - F(u) - L h\| \le \rho_u(\|h\|)\cdot\|h\| \,,
\end{equation}
where $\rho_u(\delta)\downarrow 0$ as $\delta\downarrow 0$.
If $\ell$ is a global Lipschitz constant for $F$ and $c_G$ is a global bound
for the norms $\|L\|$ of the elements $L\in G(U)$, we may assume that $\rho_u$
is globally bounded,
\begin{equation}\label{nwd.ndef.5}
\rho_u \le \ell + c_G \qquad \forall \; u\in U\,,
\end{equation}
as in the case of the Bouligand derivative.

If $G: U \toto \mathcal{L}(X,Y)$ is a Newton derivative of $F$ in $U$,
then so is every $\tilde{G}: U \toto \mathcal{L}(X,Y)$ satisfying
$\tilde{G}(u) \subset G(u)$ for all $u\in U$. In particular, every
\textbf{selector} $S: U\to L(X,Y)$ of $G$, that is, $S(u) \in G(u)$
for all $u\in U$, yields a single-valued Newton derivative of $F$ in $U$.

We now consider the following situation. 
The domain of definition $U$ of $F$ can be represented as
\begin{equation}\label{nwd.part.0}
U = \bigcup_{n\in \mathbb{N}} U_n \,,
\end{equation}
where $U_n \subset U$ are open sets with $U_n \subset U_{n+1}$ for all $n$,
and $U_0 = \emptyset$.
We want to obtain a Newton derivative of $F$ on $U$ from Newton derivatives
of $F$ on $U_n$. This can be done in the following setting.
Let $V_n\subset U$ be open sets with
\begin{equation}\label{nwd.part.01}
\overline{V}_n \subset U_n \cap V_{n+1} \quad \text{for all $n\in\mathbb{N}$,}
\quad \bigcup_{n\in\mathbb{N}} V_n = U \,.
\end{equation}
\begin{proposition}\label{nwd.part}
Let $G_n$ be a Newton derivative of $F$ on $U_n$, $n\in \mathbb{N}$,
with the remainder $\rho_{n,u}$ according to (\ref{nwd.ndef.4}).
Then in the situation just described above, the definition
\begin{equation}\label{nwd.part.1}
G(u) = G_n(u) \,,\quad \text{if $u\in \overline{V}_n \setminus \overline{V}_{n-1}$,}
\end{equation}
yields a Newton derivative $G:U\toto \mathcal{L}(X;Y)$ of $F$ on $U$
with the remainder
\begin{equation}\label{nwd.part.2}
\rho_u = \max\{\rho_{n,u},\rho_{n+1,u}\}
\quad \text{if $u\in \overline{V}_n \setminus \overline{V}_{n-1}$.}
\end{equation}
\end{proposition}
\begin{proof}
By construction,
\[
U = \bigcup_{n\in\mathbb{N}} \overline{V}_n \setminus \overline{V}_{n-1} \,,
\]
the union being disjoint. Let $u\in U$, assume that
$u\in \overline{V}_n \setminus \overline{V}_{n-1}$.
We choose $\delta > 0$ such that $B_{\delta}(u) = \{v: \|v-u\| < \delta\}$
satisfies, see (\ref{nwd.part.01}),
\[
B_{\delta}(u) \cap \overline{V}_{n-1} = \emptyset \,,\quad
B_{\delta}(u) \subset U_n \cap V_{n+1} \,.
\]
Let $h\in X$, $\|h\| < \delta$, let $L\in G(u+h)$.
If $u+h\in \overline{V}_n$, then $u+h \in \overline{V}_n \setminus \overline{V}_{n-1}$,
$u+h\in U_n$ and $L\in G_n(u+h)$, so
\[
\|F(u+h) - F(u) - Lh\| \le \rho_{n,u}(\|h\|) \|h\| \,.
\]
If $u+h \notin \overline{V}_n$, then
$u+h\in \overline{V}_{n+1} \setminus \overline{V}_{n}\subset U_{n+1}$
and $L\in G_{n+1}(u+h)$, so
\[
\|F(u+h) - F(u) - Lh\| \le \rho_{n+1,u}(\|h\|) \|h\| \,.
\]
This proves the assertions.
\end{proof}

\begin{remark}\label{nwd.partu}
{\rm
If we have $G_n(u)\subset G_{n+1}(u)$ for all $n$ and $u$, we may dispense
with the sets $V_n$ and simply define a Newton derivative $G$ of $F$ on $U$ by 
\[
G(u) = G_n(u) \,,\quad \text{if $u\in U_n \setminus U_{n-1}$.}
\]
However, in the construction of the Newton derivative of the play given below,
this property is not satisfied.
\hfill$\Box$
}
\end{remark}

%

The following result (Lemma 8.11 in \cite{IK}) shows that
Bouligand and Newton derivatives are closely related.

\begin{proposition}\label{nwd.nbou}
Let $F:U\to Y$ possess the single-valued Newton derivative
$D^NF: U \to \mathcal{L}(X,Y)$. Then $F$ is Bouligand differentiable
at $u\in U$ if and only if the limit
$\lim_{\lambda\downarrow 0} D^N F(u+\lambda h)h$ exists uniformly
w.r.t. $h\in X$ with $\|h\| = 1$. In this case,
\begin{equation}\label{nwd.nbou.1}
F'(u;h) = \lim_{\lambda\downarrow 0} D^N F(u+\lambda h)h \,.
\end{equation}
\hfill$\Box$
\end{proposition}
\section{The maximum functional}

We consider $\varphi:C[a,b]\to\real$,
\begin{equation}\label{mfu.1}
\varphi(u) = \max_{s\in [a,b]} u(s) \,.
\end{equation}
The functional $\varphi$ is convex, positively 1-homogeneous
and globally Lipschitz continuous with Lipschitz constant 1,
w.r.t. the maximum norm on $C[a,b]$.
By convex analysis, it is directionally (and thus, Hadamard) differentiable.
An explicit formula for the directional derivative is given by
(see e.g. \cite{Gir} for a direct proof)
\begin{equation}\label{mfu.2}
\varphi'(u;h) = \max_{s\in M(u)} h(s) \,,
\end{equation}
where
\begin{equation}\label{mfu.3}
M(u) = \{\tau\in [a,b],\,u(\tau) = \varphi(u)\}
\end{equation}
is the set where $u$ attains its maximum.

Let us denote the dual of $C[a,b]$ by $C[a,b]^*$; it consists of all signed
regular Borel measures on $[a,b]$. The subdifferential of $\varphi$ is defined
as usual as the set-valued mapping $\partial\varphi:C[a,b]\toto C[a,b]^*$
given by
\begin{equation}\label{mfu.4}
\partial\varphi(u) = \{ \mu: \text{$\mu\in C[a,b]^*$,
$\varphi(v) - \varphi(u) \ge \langle \mu, v-u \rangle$ for all $v\in C[a,b]$}\} \,.
\end{equation}
It is not difficult to check that
\begin{equation}\label{mfu.5}
\partial\varphi(u) = \{ \mu: \text{$\mu\in C[a,b]^*$,
supp$(\mu) \subset M(u)$, $\mu \ge 0$, $\|\mu\| = 1$}\} \,.
\end{equation}
In particular, if $u$ has a unique maximum at $r\in [a,b]$, that is, $M(u) = \{r\}$,
then $\partial\varphi(u) = \{\delta_r\}$, where $\delta_r$ denotes
the Dirac delta at $r$.

A side remark (we will not use this):
the directional derivative is linked to the subdifferential by the
``max formula'' (see \cite{BaCo}, Theorem 17.19, for the Hilbert space case)
\[
\varphi'(u;h) = \max_{\mu\in\partial\varphi(u)} \langle \mu,h\rangle \,.
\]
The subdifferential is a natural candidate for a Newton derivative of a
convex functional.
However, the subdifferential of $\varphi:C[a,b]\to\real$ is not a Newton
derivative of $\varphi$, and $\varphi$ is not Bouligand differentiable.
The following example shows that this is true even if we restrict
$\varphi$ to $W^{1,1}(a,b)$.

Here and in the sequel we use the norm
\[
\|u\|_{W^{1,p}} = |u(a)| + \|u'\|_p = |u(a)| + \Big( \int_a^b |u'(s)|^p\,ds
\Big)^{1/p} \,,\quad 1 \le p < \infty \,.
\]
\begin{example}\label{mfu.notbou} {\rm
Consider $u:[0,1]\to\real$ defined by $u(s) = 1 - s$. We have $\varphi(u) = 1$
and $M(u) = \{0\}$. Define $h_\lambda:[0,1]\to\real$ for $\lambda > 0$ by
\begin{equation}\label{mfu.notbou.1}
h_\lambda(s) = \begin{cases} 2s \,,& s\le \lambda \,, \\ 2\lambda \,,& s > \lambda \,.
\end{cases}
\end{equation}
Then the function $u + h_\lambda$ attains its maximum at $s = \lambda$, and
\[
\|h_\lambda\|_{1,1} = 2\lambda \,,\quad \varphi(u + h_\lambda) = 1 + \lambda
\,,\quad \varphi'(u;h_\lambda) = \max_{s\in M(u)} h_\lambda(s) = h_\lambda(0) = 0 \,.
\]
Consequently, $\|h_\lambda\|_{1,1} \to 0$ but
\begin{equation}\label{mfu.notbou.2}
\frac{|\varphi(u + h_\lambda) - \varphi(u) - \varphi'(u;h_\lambda)|}%
{\|h_\lambda\|_{1,1}} = \frac{\lambda}{2\lambda} = \frac{1}{2} \,.
\end{equation}
Thus, $\varphi$ is not Bouligand differentiable at $u$ on
$X = W^{1,1}(a,b)$.
Moreover, setting $\Phi = (\partial\varphi)|X$ we obtain
\[
M(u+h_\lambda) = \{\lambda\} \,,\quad \Phi(u+h_\lambda) = \{\delta_\lambda\} \,,\quad
\Phi(u+h_\lambda)h_\lambda = h_\lambda(\lambda) = 2\lambda \,,
\]
so
\begin{equation}\label{mfu.notbou.3}
\frac{|\varphi(u + h_\lambda) - \varphi(u) - \Phi(u+h_\lambda)h_\lambda|}%
{\|h_\lambda\|_{1,1}} = \frac{\lambda}{2\lambda} = \frac{1}{2} \,.
\end{equation}
Thus, $\Phi$ is not a Newton derivative of $\varphi$ on $W^{1,1}(0,1)$.
As $\|h_\lambda\|_\infty = \|h_\lambda\|_{1,1}$ (or due to the
embedding $W^{1,1}\to C$), the same is true on $C[0,1]$.
\hfill $\Box$
}
\end{example}
We will show that $\Phi$ is a Newton derivative of $\varphi$ on
$C^{0,\alpha}[a,b]$ for every $\alpha > 0$, endowed with the norm
\begin{equation}\label{mfu.hoenorm.1}
\|u\|_{C^{0,\alpha}} = |u(a)| + |u|_{C^{0,\alpha}} \,,\quad
|u|_{C^{0,\alpha}} =
\sup_{t,s\in [a,b] \atop s\neq t} \frac{|u(t) - u(s)|}{|t-s|} \,.
\end{equation}
We set $B_\veps = (-\veps,\veps)$.
\begin{lemma}\label{bmf.mch}
The mapping $M:C[a,b]\toto [a,b]$ is upper semicontinuous, that is,
for every $u\in C[a,b]$ and every $\veps > 0$ there exists $\delta > 0$
such that for every $h\in C[a,b]$
\begin{equation}\label{bmf.mch.1}
\|h\|_\infty < \delta \quad\Rightarrow\quad M(u+h) \subset  M(u) + B_\veps \,.
\end{equation}
\end{lemma}
\begin{proof}
By contradiction. Assume that $u\in C[a,b]$ and $\veps > 0$ are such that
for all $n\in\mathbb{N}$ there exist $h_n\in C[a,b]$ with
$\|h_n\|_\infty < \frac{1}{n}$ and $M(u+h_n) \not\subset M(u) + B_\veps$.
Let $t_n\in M(u+h_n)$ with $d(t_n,M(u)) \ge \veps$. Passing to a
subsequence we get $t_n\to t\in [a,b]$, $t\notin M(u)$.
On the other hand, $u(t_n) + h_n(t_n) = \varphi(u+h_n)$.
Letting $n\to\infty$ yields $u(t) = \varphi(u)$, so $t\in M(u)$,
a contradiction.
\end{proof}
\par
For a function $f:I\to\real$, $I$ being an interval, we denote its oscillation
on $I$ by
\begin{equation}\label{bmf.osc.1}
\osc_I(f) =  \sup \{ |f(t) - f(s)|: t,s\in I \} \,,
\end{equation}
and its modulus of continuity by
\begin{equation}\label{bmf.4}
\omega_I (f;\veps) = \sup \{ |f(t) - f(s)|: t,s\in I \,,\, |t-s|\le\veps\} \,.
\end{equation}
When $I = [a,b]$, we simply write $\osc(f)$ and $\omega(f;\veps)$.
\begin{lemma}\label{nmf.estbelow}
Let $u,h\in C[a,b]$, $\mu\in \partial\varphi(u+h)$. Then
\begin{equation}\label{nmf.estbelow.1}
\varphi'(u;h) \le \varphi(u+h) - \varphi(u) \le \langle \mu , h \rangle \,.
\end{equation}
Let moreover be $\veps > 0$ such that
\begin{equation}\label{nmf.estbelow.11}
M(u+h) \subset M(u) + B_\veps \,.
\end{equation}
Then we have
\begin{equation}\label{nmf.estbelow.2}
\langle \mu , h \rangle - \varphi'(u;h) \le \sup_{|s-r|\le \veps} |h(r) - h(s)|
= \omega(h;\veps) \,.
\end{equation}
\end{lemma}
\begin{proof}
The first inequality in (\ref{nmf.estbelow.1}) holds since $\varphi$ is convex;
as $\varphi(u) - \varphi(u+h)\ge \langle \mu, -h\rangle$, the second
inequality follows.
Now assume that (\ref{nmf.estbelow.11}) holds. Recalling (\ref{mfu.5}),
given $r\in{\rm supp}(\mu) \subset M(u+h)$ we find an $s_r\in M(u)$
with $|r - s_r| < \veps$, so
\[
h(r) - \varphi'(u;h) = h(r) - \max_{s\in M(u)} h(s) \le h(r) - h(s_r) \le
\omega(h;\veps) \,.
\]
Integrating both sides of this inequality over $r\in [a,b]$ with respect to
$\mu$ yields (\ref{nmf.estbelow.2}).
\end{proof}

For the modulus of continuity, we have
\begin{equation}\label{nmf.mod.1}
\omega(h;\veps) \le |h|_{C^{0,\alpha}} \veps^\alpha \,,\quad
\omega(h;\veps) \le \|h'\|_{L^p} \veps^{1 - 1/p} \,.
\end{equation}
\begin{proposition}\label{nmf.nbder}
Let $X = C^{0,\alpha}[a,b]$ or $X = W^{1,p}(a,b)$,
with $0 < \alpha \le 1$ resp. $1 < p \le \infty$.
Then the set-valued mapping $\Phi = (\partial\varphi)|X$ given in
(\ref{mfu.5}) is a globally bounded Newton derivative of the
maximum functional $\varphi$ on $X$.
In particular, for every $u\in X$ there exists a nondecreasing and bounded
$\rho_u: \real_+\to\real_+$ such that $\rho_u(\delta)\to 0$ as $\delta\to 0$,
$\rho_u$ is bounded independently from $u$, and
\begin{equation}\label{nmf.nbder.1}
|\varphi(u+h) - \varphi(u) - Lh| \le
\begin{cases}\rho_u(\|h\|_\infty)|h|_{C^{0,\alpha}} \\
\rho_u(\|h\|_\infty)\|h'\|_{L^p}
\end{cases}
\end{equation}
respectively, for every $h\in X$ and every $L\in \Phi(u+h)$.

Moreover, $\varphi$ is Bouligand differentiable on $X$, and for every $u\in X$
\begin{equation}\label{nmf.nbder.2}
|\varphi(u+h) - \varphi(u) - \varphi'(u;h)| \le
\begin{cases}\rho_u(\|h\|_\infty)) |h|_{C^{0,\alpha}} \\ \rho_u(\|h\|_\infty)\|h'\|_{L^p}
\end{cases}
\end{equation}
respectively, for every $h\in X$.
\end{proposition}
\begin{proof}
We consider the case $X = C^{0,\alpha}[a,b]$.
Let $u\in X$ be given, let
\[
\veps_u(\delta) = \inf\{\veps: M(u+B_\delta)\subset M(u) + B_\veps\}
\]
for $\delta > 0$. Then $\veps_u$ is increasing.
As $M$ is upper semicontinous by Lemma \ref{bmf.mch},
we have $0 < \veps_u(\delta) \to 0$ as $\delta \to 0$,
According to (\ref{nmf.estbelow.1}) and (\ref{nmf.estbelow.2}),
for $h\in X$ and $L = \mu\in\Phi(u+h)$ we get
\[
|\varphi(u+h) - \varphi(u) - Lh| \le \omega(h;\veps_u(\|h\|_\infty)
\le \veps_u(\|h\|_\infty)^\alpha\cdot |h|_{C^{0,\alpha}} \,.
\]
Setting $\rho_u(\delta) = \veps_u(\delta)^\alpha$, (\ref{nmf.nbder.1})
follows for the H\"older case. Since $\|L\|_{C\to\real} = 1$, we have
$\|L\|_{C^{0,\alpha}\to\real} \le c_\alpha$ and
\[
|\varphi(u+h) - \varphi(u) - Lh| \le 2\|h\|_\infty \le 2c_\alpha \|h\|_\infty \,,
\]
where $c_\alpha$ denotes the norm of the embedding $C^{0,\alpha}\to C$.
Thus, $c_\alpha$ is a global bound for $\Phi$,
and $2c_\alpha$ furnishes a global bound for $\rho_u$.

The proof for the case $X =  W^{1,p}(a,b)$ is analogous. (One might also
refer to Morrey's embedding theorem which implies that 
$W^{1,p}(a,b)$ is continuously embedded into $C^{0,\alpha}[a,b]$ for
$\alpha \le 1 - 1/p$.)
\end{proof}
\par
Note that the estimates (\ref{nmf.nbder.1}) and (\ref{nmf.nbder.2})
are slightly stronger than required for Newton and Bouligand differentiability
(the factor $\rho_u(\|h\|_X)$ instead of $\rho_u(\|h\|_\infty)$, as well as
the norms instead of the seminorms, would suffice).
This strenghtening is motivated by applications to partial differential equations.

\section{The cumulated maximum}

We define the cumulated maximum of a function $u\in C[a,b]$ as
\begin{equation}\label{am.1}
\varphi_t(u) = \max_{s\in [a,t]} u(s) \,,\quad t\in [a,b] \,.
\end{equation}
Setting
\begin{equation}\label{am.2}
(Fu)(t) = \varphi_t(u)
\end{equation}
we obtain an operator
\begin{equation}\label{am.3}
F: C[a,b] \to C[a,b] \,.
\end{equation}
The function $Fu$ is nondecreasing for every $u\in C[a,b]$. Since
\[
|\varphi_t(u) - \varphi_t(v)| \le \max_{s\in [a,t]} |u(s) - v(s)| \,,
\quad \text{for all $u,v\in C[a,b]$,}
\]
we have
\begin{equation}\label{am.4}
\|Fu - Fv\|_{\infty,t} \le \|u - v\|_{\infty,t} \,,\quad
\text{for all $u,v\in C[a,b]$, $t\in [a,b]$.}
\end{equation}
Here and in the following we use the notation
\begin{equation}\label{am.41}
\|u\|_{\infty,t} = \sup_{s\le t} |u(s)| \,.
\end{equation}
For any fixed $t\in [a,b]$, the directional derivative of
$\varphi_t:C[a,b]\to\real$ given in (\ref{mfu.2}) yields that,
for all $u,h\in C[a,b]$,
\begin{equation}\label{am.pwd.1}
F^{PD}(u;h)(t) :=
\lim_{\lambda\downarrow 0} \frac{(F(u+\lambda h))(t) - (Fu)(t)}{\lambda}
= \varphi_t'(u;h) = \max_{s\in M(u,t)} h(s) \,,
\end{equation}
where
\begin{equation}\label{am.pwd.2}
M(u,t) = \{\tau: \tau\in [a,t],\,u(\tau) = \varphi_t(u)\}
\end{equation}
is the set where $u$ attains its maximum on $[a,t]$. As in \cite{BK},
we call \textbf{pointwise directional derivative of {\boldmath $F$}}
the function $F^{PD}(u;h):[a,b]\to \real$ obtained in this manner.

Example 4.3 in \cite{BK} shows that
the function $F^{PD}(u;h):[a,b]\to\real$ does not need to be continuous
even though $u$ and $h$ are; so $F:C[a,b]\to C[a,b]$ is not
directionally differentiable. When this happens, the difference
quotients
\[
\frac{F(u+\lambda h) - Fu}{\lambda}
\]
do not converge uniformly to $F^{PD}(u;h)$. They do, on the other hand,
converge in $L^r(a,b)$ for every $r < \infty$, as they are uniformly bounded
by $\|h\|_\infty$. As a consequence, $F:C[a,b]\to L^r(a,b)$ is
Hadamard differentiable (\cite{BK}). In order to obtain Bouligand or Newton
differentiability, as in the case of the maximum functional one
has to strengthen the norm in the domain space. Indeed, the functions
from Example \ref{mfu.notbou} can be used to show that $F$ is not
Bouligand differentiable on $C[a,b]$.

\textbf{Bouligand differentiability of the cumulated maximum.}
Let again $X$ stand for $C^{0,\alpha}[a,b]$ with $0 < \alpha\le 1$,
or for $W^{1,p}(a,b)$ with $1 < p \le \infty$.
We want to prove that $F:X\to L^q(a,b)$ is Bouligand
differentiable for $1\le q < \infty$ with the improved
remainder estimate as in Proposition \ref{nmf.nbder}. For this, we
have to show that
\begin{equation}\label{am.bourem.1}
\rho_u^F(\delta) := \sup_{\|h\|_\infty \le \delta}
\frac{\|F(u+h) - F(u) - F'(u;h)\|_{L^q}}{\|h\|_{X}} \to 0
\qquad \text{as $\delta\to 0$.}
\end{equation}

\begin{proposition}\label{am.bounew}
The cumulated maximum $F:X\to L^q(a,b)$ is Bouligand
differentiable for every $q < \infty$, and $F' = F^{PD}$.
Moreover,
\begin{equation}\label{am.bounew.1}
\|F(u+h) - F(u) - F'(u;h)\|_{L^q} \le
\rho_u^F(\|h\|_\infty)\cdot \|h\|_{X} \,,
\end{equation}
and $\rho_u^F(\delta) \to 0$ as $\delta \to 0$.
In addition, $\rho_u^F$ is bounded uniformly in $u$.
\end{proposition}

\begin{proof}
Assume that (\ref{am.bourem.1}) does not hold. Then there exists $\veps > 0$
and a sequence $\{h_n\}$ in $X$ with $\|h_n\|_\infty \to 0$
and
\begin{equation}\label{am.bounew.2}
\veps \|h_n\|_{X} \le \|F(u+h_n) - F(u) - F^{PD}(u;h_n)\|_{L^q}
= \left( \int_a^b d_n(t)^q\,dt \right)^{1/q} \,,
\end{equation}
where
\[
d_n(t) = |\varphi_t(u+h_n) - \varphi_t(u) - \varphi_t'(u;h_n)| \,.
\]
Setting $\rho_n = d_n/\|h_n\|_{X}$ we have $\rho_n(t)\to 0$ pointwise,
because $\varphi_t:X\to\real$ is Bouligand differentiable for every $t$
by Proposition \ref{nmf.nbder}, with the remainder estimate (\ref{nmf.nbder.2}).
Since moreover $\{\rho_n\}$ is uniformly bounded,
by dominated convergence $\|\rho_n\|_{L^q}\to 0$ which contradicts (\ref{am.bounew.2}).
Therefore $F$ is Bouligand differentiable and $F' = F^{PD}$.
The global bound on $\rho_u^F$ follows from the estimate
$\|F(u+h) - F(u) - F'(u;h)\|_{\infty} \le 2\|h\|_\infty$ combined with the
embedding constants.
\end{proof}

\textbf{Newton differentiability of the cumulated maximum.}
A Newton derivative of the cumulated maximum is constructed from the
Newton derivative of the maximum functional given in the previous section.
Its elements $L$ will have the form $(Lh)(t) = \langle \mu^t,h\rangle$,
where $\mu^t$ belongs to the Newton derivative $\Phi^t$ of $\varphi_t$.
In order that $Lh$ becomes a measurable function, the measures $\mu^t$
are constructed from measurable selectors of the family $\{\Phi^t\}$.

We first analyze the mapping $M:C[a,b]\times [a,b]\toto [a,b]$
\begin{equation}\label{mam.mtu.1}
M(u,t) = \{\tau: \tau\in [a,t],\, u(\tau) = \varphi_t(u)\} \,.
\end{equation}
The sets $M(u,t)$ are compact nonempty subsets of $[a,b]$,
and $M(u,a) = \{a\}$.

\begin{lemma}\label{mam.mtusc}
The set-valued mapping $M$ is upper semicontinuous and measurable.
\end{lemma}

\begin{proof}
To prove that $M$ is upper semicontinuous according to Definition \ref{svm.uscdef},
let $A\subset [a,b]$ be closed, and let $(u_n,t_n)$ be a sequence in $M^{-1}(A)$
with $u_n\to u\in C[a,b]$ and $t_n \to t\in [a,b]$. In order to show that
$(u,t)\in M^{-1}(A)$, let $\tau_n\in A$ such that $\tau_n\in M(u_n,t_n)$,
thus $u_n(\tau_n) = \varphi_{t_n}(u_n)$. Passing to a subsequence we have
$\tau_n\to \tau\in A$ since $A$ is closed. Moreover,
$\tau\le t$, $u_n(\tau_n)\to u(\tau)$ and
\[
\varphi_{t_n}(u_n) = (\varphi_{t_n}(u_n) - \varphi_{t_n}(u)) + \varphi_{t_n}(u)
\to \varphi_t(u)
\]
by (\ref{am.4}) and
since $t\mapsto \varphi_t(u)$ is continuous. Therefore $u(\tau) = \varphi_t(u)$
and $\tau\in M(u,t)$. Thus $M$ is upper semicontinuous. It now follows from
Proposition 6.2.3 in \cite{PK} that $M$ is measurable.
\end{proof}

The set-valued mapping $M$ possesses a dense sequence of measurable selectors.
\begin{proposition}\label{mam.mtsel}
There exists a sequence $\{f_n\}$ of measurable selectors of $M$ such that
\begin{equation}\label{mam.mtsel.1}
M(u,t) = \overline{\{f_n(u,t): n\in\mathbb{N}\}} \,,\quad
\text{for all $u\in C[a,b]$, $t\in [a,b]$} \,.
\end{equation}
In particular $\max M(u,t) = \sup_n f_n(u,t)$ and $\min M(u,t) = \inf_n f_n(u,t)$
are measurable selectors of $M$.
\end{proposition}

\begin{proof}
This is a consequence of Theorem 6.3.18 in \cite{PK}, as $[a,b]$ is a complete
separable metric space.
\end{proof}

We consider the mapping $\Phi: C[a,b]\times [a,b]\toto C[a,b]^*$,
\begin{equation}\label{mam.phit.1}
\Phi(u,t) = \{\nu\in C[a,b]^*: {\rm supp}(\nu)\subset M(u,t), \nu\ge 0, \|\nu\|=1\} \,.
\end{equation}
The following facts are well known. The closed unit ball
$K$ in $C[a,b]^*$, endowed with the weak star topology, is compact (hence
complete), metrizable and separable. The sets $\Phi(u,t)$ are nonempty
convex and weak star compact subsets of $K$
(note that for $\nu\ge 0$ we have $\|\nu\| = \langle \nu, 1\rangle$).
Moreover,
\begin{gather}\label{mam.phit.2}
\Phi(u,a) = \{\delta_a\} \,,
\\ \label{mam.phit.3}
M(u+c,t) = M(u,t) \,,\quad \Phi(u+c,t) = \Phi(u,t) \quad
\text{for all $c\in\real$,}
\\ \label{mam.phit.4}
(\Phi(u,t))(c) = \{c\} \quad \text{for all $c\in\real$.}
\end{gather}
\begin{lemma}\label{mam.suppcl}
Let $\{u_n\}$, $\{t_n\}$, $\{\nu_n\}$ be sequences in $C[a,b]$, $[a,b]$
and $C[a,b]^*$ respectively, with $u_n\to u$, $t_n\to t$ and
$\nu_n \wsto \nu$, let ${\rm supp}(\nu_n)\subset M(u_n,t_n)$
for all $n\in\mathbb{N}$.
Then ${\rm supp}(\nu)\subset M(u,t)$.
\end{lemma}

\begin{proof}
Let $f\in C_0^\infty(\real\setminus M(u,t))$. We have to show that
$\langle \nu,f\rangle = 0$. Let
\[
\veps = \inf\{|s-\tau|: s\in {\rm supp}(f), \tau\in M(u,t)\} \,.
\]
We have $\veps > 0$ because the sets ${\rm supp}(f)$ and $M(u,t)$
are disjoint and compact. Since $M$ is upper semicontinuous by
Proposition \ref{mam.mtusc}, we may choose $N\in\mathbb{N}$ such that
$M(u_n,t_n) \subset M(u,t) + B_{\veps/2}$ holds for all $n \ge N$.
Then ${\rm supp}(f)\cap M(u_n,t_n) = \emptyset$ and thus
$\langle \nu_n,f\rangle = 0$ for all $n\ge N$. Passing to the limit
$n\to\infty$ we arrive at $\langle \nu,f\rangle = 0$.
\end{proof}

\begin{proposition}\label{mam.phitusc}
The mapping $\Phi: C[a,b]\times [a,b]\toto C[a,b]^*$ defined in
(\ref{mam.phit.1}) is upper semicontinuous, thus measurable.
\end{proposition}

\begin{proof}
Let $A\subset C[a,b]^*$ be weak star closed. We have to show that
$\Phi^{-1}(A)$ is closed. To this end, let $\{(u_n,t_n)\}$ be a
sequence in $\Phi^{-1}(A)$ with $u_n\to u$ in $C[a,b]$ and $t_n\to t$
in $[a,b]$. Let $\nu_n\in \Phi(u_n,t_n)$, so $\nu_n\in A$ as well
as $\nu_n\ge 0$, $\|\nu_n\| = 1$ and ${\rm supp}(\nu_n)\subset M(u_n,t_n)$
for all $n\in\mathbb{N}$. For some subsequence, we have
$\nu_{n_k}\wsto \nu$ with $\nu\ge 0$, $\|\nu\| = 1$ and $\nu\in A$. By
Proposition \ref{mam.suppcl}, ${\rm supp}(\nu)\subset M(u,t)$.
Thus, $(u,t)\in\Phi^{-1}(A)$ and the proof is complete.
\end{proof}

\begin{proposition}\label{mam.phitsel}
There exists a sequence $\{\mu_n\}$ of measurable selectors of $\Phi$ such
that
\begin{equation}\label{mam.phitsel.1}
\Phi(u,t) = \overline{\{\mu_n(u,t): n\in\mathbb{N}\}} \,,\quad
\text{for all $u\in C[a,b]$, $t\in [a,b]$} \,,
\end{equation}
the closure being taken w.r.t. the weak star topology.
\end{proposition}

\begin{proof}
This follows from Theorem 6.3.18 in \cite{PK}, as the unit
ball in $C[a,b]^*$ is a complete separable metrizable space w.r.t. the
weak star topology.
\end{proof}

\begin{lemma}\label{mam.l}
Let $\mu$ be a measurable selector of $\Phi$. Then
\[
(Lh)(t) = \langle \mu(u,t),h \rangle
\]
defines an element $L\in \mathcal{L}(C[a,b];L^\infty(a,b))$ with
$\|L\| = 1$ and
\begin{equation}\label{mam.l.1}
\|Lh\|_{\infty,t} \le \|h\|_{\infty,t} \,.
\end{equation}
\end{lemma}

\begin{proof}
For every $u,h\in C[a,b]$, the mapping $t\mapsto \langle \mu(u,t),h \rangle$
is measurable and satisfies
$|\langle \mu(u,s),h \rangle| \le \|h\|_{\infty,t}$ for all $s\le t$,
since $\mu(u,s)$ has support in $[a,s]$.
Thus, $L$ is well-defined, $\|L\|\le 1$ and (\ref{mam.l.1}) holds. 
As $\mu\ge 0$ and $L(1) = 1$, we have $\|L\| = 1$.
\end{proof}
\par
Let $X$ again denote any one of the spaces $C^{0,\alpha}[a,b]$ for
$0 < \alpha \le 1$ or $W^{1,p}(a,b)$ for $1 < p\le\infty$.
\begin{proposition}\label{am.newder}
Let $S_\Phi$ be the set of all measurable selectors of $\Phi$,
let $q\in [1,\infty)$.
The set-valued mapping $G:X\toto \mathcal{L}(X,L^q(a,b))$,
\begin{equation}\label{am.newder.1}
G(u) = \{L: \text{$(Lh)(t) = \langle \mu(u,t),h \rangle$, $\mu\in S_\Phi$}\}
\end{equation}
defines a Newton derivative of the cumulated maximum
$F:X\to L^q(a,b)$ with
\begin{equation}\label{am.newder.2}
\|F(u+h) - F(u) - Lh\|_{L^q} \le \rho_u^G(\|h\|_\infty) \cdot\|h\|_{X} \,,
\end{equation}
for all $L\in G(u+h)$. Here, $\rho_u^G:\real_+\to\real_+$ is an increasing
function with $\rho_u^G(\delta)\to 0$ as $\delta\to 0$, bounded independently
from $u$.
\end{proposition}

\begin{proof}
Fix $u\in X$. For $h\in  X$ we define
\[
d(h,t) = \sup_{\mu^t\in \Phi(u+h,t)}
|\varphi_t(u+h) - \varphi_t(u) - \langle \mu^t,h\rangle| \,.
\]
Let $\{\mu_k\}$ be a sequence of measurable selectors of $\Phi$ according to
Proposition \ref{mam.phitsel}, set
\[
d_k(h,t) = |\varphi_t(u+h) - \varphi_t(u) - \langle \mu_k(u+h,t),h\rangle|
\]
Then $d(h,t) = \sup_k d_k(h,t)$ by (\ref{mam.phitsel.1}), and therefore
the mapping $t\mapsto d(h,t)$ is measurable. Moreover,
\[
\sup_{L\in G(u+h)} \|F(u+h) - F(u) - Lh\|_{L^q} =
\left( \int_a^b d(h,t)^q\,dt \right)^{1/q} =: d^G(h) \,.
\]
The remainder of the proof is analogous to that of Proposition \ref{am.bounew}.
We define
\begin{equation}\label{am.newder.5}
\rho_u^G(\delta) = \sup_{\|h\|_\infty \le \delta} \frac{d^G(h)}{\|h\|_{X}} \,.
\end{equation}
Assume that $\lim_{\delta\to 0} \rho_u^G(\delta) = 0$ does not hold.
Then there exist $\veps > 0$ and a sequence $\{h_n\}$ in $X$
with $\|h_n\|_\infty \to 0$ and
\begin{equation}\label{am.newder.6}
\veps \|h_n\|_{X} \le\left( \int_a^b d(h_n,t)^q\,dt \right)^{1/q} \,.
\end{equation}
Since $\Phi(\cdot,t)$ is a Newton derivative of $\varphi_t$, we have
$\rho_n(t) = d(h_n,t)/\|h_n\|_X \to 0$ pointwise in $t$ as $n\to\infty$.
Moreover, $\rho_n$ is uniformly bounded. Applying dominated convergence,
we arrive at a contradiction to (\ref{am.newder.6}). The global boundedness
of $\rho_u^G$ follows from the estimate $\|F(u+h)-F(u)-Lh\|_\infty \le 2\|h\|_\infty$.
\end{proof}
\par
Proposition \ref{mam.mphisel} below shows that the set $S_\Phi$ is large enough
to approximate the whole range of $\Phi$.
\begin{lemma}\label{mam.mphi}
Let $f:C[a,b]\times [a,b]\to [a,b]$ be a measurable selector of $M$.
Then
\begin{equation}\label{mam.mphi.1}
\mu(u,t) = \delta_{f(u,t)}
\end{equation}
defines a measurable selector $\mu:C[a,b]\times [a,b]\to C[a,b]^*$
of $\Phi$.
\end{lemma}

\begin{proof}
For each $v\in C[a,b]$, the mapping $s\mapsto v(s) = \langle \delta_s,v\rangle$
is continuous from $[a,b]$ to $\real$. Thus, the mapping $s\mapsto \delta_s$
is weak star continuous from $[a,b]$ to $C[a,b]^*$, and consequently
(\ref{mam.mphisel.1}) defines a measurable mapping.
\end{proof}

\begin{proposition}\label{mam.mphisel}
Let $\{f_n\}$ be a sequence of measurable selectors of $M$ such
that
\begin{equation}\label{mam.mphisel.1}
M(u,t) = \overline{\{f_n(u,t): n\in\mathbb{N}\}} \,,\quad
\text{for all $u\in C[a,b]$, $t\in [a,b]$} \,.
\end{equation}
Taking all rational convex combinations of the mappings
$(u,t)\mapsto \delta_{f_n(u,t)}$ we obtain a sequence $\{\mu_n\}$
of measurable selectors of $\Phi$ such that
\begin{equation}\label{mam.mphisel.2}
\Phi(u,t) = \overline{\{\mu_n(u,t): n\in\mathbb{N}\}} \,,\quad
\text{for all $u\in C[a,b]$, $t\in [a,b]$} \,,
\end{equation}
the closure being taken w.r.t. the weak star topology.
\end{proposition}

\begin{proof}
Let $u\in C[a,b]$ and $t\in [a,b]$ be given.
The set $D = \{f_n(u,t): n\in\mathbb{N}\}$ is a countable dense
subset of $M(u,t)$.
The set of all convex combinations with rational coefficients
of elements of the set $\{\delta_\tau: \tau\in D\}$
then is dense in $\Phi(u,t)$ w.r.t. the weak star topology.
\end{proof}

\section{The chain rule}

In the following sections we will see that the play operator can be 
represented as a finite composition of cumulated maxima and positive
part mappings. The Newton differentiability of these mappings will
imply Newton differentiability of the play, by virtue of the chain
rule. It is a standard result that the chain rule is valid for Newton
derivatives, see Proposition A.1 in \cite{HK09} for the single-valued
and Proposition 3.8 in \cite{Ulb11} for the set-valued case.

As a result of investigating the maximum and the cumulated maximum,
we have seen above that these operators satisfy a slightly stronger version of
Newton and Bouligand differentiability. For the cumulated maximum
$F: X\to Y$ with $X = W^{1,p}(a,b)$ or $C^{0,\alpha}[a,b]$ and
$Y = L^r(a,b)$, we have constructed a Newton derivative
$G: X\toto \mathcal{L}(X;Y)$ with a remainder estimate
\begin{equation}\label{rr.1}
\sup_{L\in G(u+h)}
\|F(u+h) - F(u) - Lh\|_Y \le \rho_u(\|h\|_{\tilde{X}})\cdot \|h\|_X \,,
\end{equation}
where $\tilde{X} = C[a,b]$, endowed with the maximum norm. 
The purpose of this section is to extend the chain rule to this situation, 
for Newton as well as for Bouligand derivatives.

We consider the following setting. 
\begin{assumption}\label{rr.ass}\hfill\\
(i)
$X,Y,Z$ are normed spaces, $U\subset X$ and $V\subset Y$ are open.
$F_1:U\to Y$ and $F_2:V\to Z$ with $F_1(U)\subset V$ are locally
Lipschitz. 
\hfill\\ (ii)
$\tilde{X}$ and $\tilde{Y}$ are normed spaces with continuous
embeddings $X\subset \tilde{X}$ and $Y\subset \tilde{Y}$.
\hfill\\ (iii)
$G_1:U\toto\mathcal{L}(X;Y)$ and $G_2:V\toto\mathcal{L}(Y;Z)$ satisfy,
for every $u\in U$ and $v\in V$,
\begin{equation}\label{rr.ass.1}
\sup_{L_1\in G_1(u+h)}
\|F_1(u+h) - F_1(u) - L_1h\|_Y \le \rho_{1,u}(\|h\|_{\tilde{X}})\cdot \|h\|_X
\end{equation}
for every $h\in X$ with $u+h\in U$,
\begin{equation}\label{rr.ass.2}
\sup_{L_2\in G_2(v+k)}
\|F_2(v+k) - F_2(v) - L_2k\|_Z \le \rho_{2,v}(\|k\|_{\tilde{Y}})\cdot \|k\|_Y
\end{equation}
for every $k\in Y$ with $v+k\in V$, with functions
$\rho_{1,u},\rho_{2,v}:\real_+\to\real_+$ satisfying 
$\rho_{1,u}(\delta)\downarrow 0$ and $\rho_{2,v}(\delta)\downarrow 0$
for $\delta\downarrow 0$.
\hfill\\ (iv) 
$F_1:(U,\|\cdot\|_{\tilde{X}})\to (V,\|\cdot\|_{\tilde{Y}})$ is
continuous.
\hfill\\ (v)
$G_2$ is locally bounded on $(V,\|\cdot\|_Y)$.
\end{assumption}
Since $\rho_{1,u}(\|h\|_{\tilde{X}}) \le \rho_{1,u}(c\|h\|_{X})$ for some constant
$c$, part (iii) of the assumption implies that $G_1$ and $G_2$ are Newton derivatives 
for $F_1$ in U and $F_2$ in $V$, respectively.
Note also that the assumption ``$G_2$ locally bounded'' already implies that
$F_2$ is locally Lipschitz.

In the special case $\tilde{X} = X$ and $\tilde{Y} = Y$, (\ref{rr.ass.1}) and
(\ref{rr.ass.2}) reduce to the standard remainder form (\ref{nwd.ndef.4}), and
part (iv) of the assumption is implied by part (i);
the following result then reduces to the standard chain rule for Newton derivatives.
\begin{proposition}[Refined Chain Rule, Newton Derivative]\label{rr.newcomp}
\hfill\\
Let Assumption \ref{rr.ass} hold.
Then
\begin{equation}\label{rr.newcomp.1}
\begin{split}
G:U\toto\mathcal{L}(X;Z) \qquad \qquad \qquad
\\
G(u) = \{L_2 \circ L_1: L_1\in G_1(u), L_2\in G_2(F_1(u)) \}
\end{split}
\end{equation}
is a Newton derivative of $F = F_2\circ F_1$ in $U$ which satisfies,
for every $u\in U$,
\begin{equation}\label{rr.newcomp.100}
\sup_{L\in G(u+h)}
\|F(u+h) - F(u) - Lh\|_Z \le \rho_u(\|h\|_{\tilde{X}})\cdot \|h\|_X 
\end{equation}
for every $h\in X$ with $u+h\in U$, where $\rho_u:\real_+\to\real_+$
is a function with $\rho_u(\delta)\downarrow 0$ for $\delta\downarrow 0$.
\end{proposition}
\begin{proof}
Let $u\in U$, $h\in X$ with $u+h\in U$, set $k = F_1(u+h) - F_1(u)$.
Let $L_1\in G_1(u+h)$, $L_2\in G_2(F_1(u+h)) = G_2(F_1(u) + k)$.
By the triangle inequality,
\begin{equation}\label{rr.newcomp.2}
\begin{split}
&\|(F_2\circ F_1)(u+h) - (F_2\circ F_1)(u) - (L_2\circ L_1)h\|_Z
\\ &\qquad \le
\|F_2(F_1(u) + k) - F_2(F_1(u)) - L_2 k\|_Z + \|L_2(k - L_1 h)\|_Z
\end{split}
\end{equation}
Since $G_2$ is locally bounded, there exists a $C>0$ such that for
sufficiently small $\|h\|_X$ we have $\|L_2\|\le C$ for all
$L_2\in G_2(F_1(u+h))$. Consequently, for all such $h$ and $L_2$,
and for all $L_1\in G_1(u+h)$ we have by (\ref{rr.ass.1})
\begin{equation}\label{rr.newcomp.3}
\|L_2(k - L_1 h)\|_Z \le C \|F_1(u+h) - F_1(u) - L_1 h\|_Y
\le C \rho_{1,u}(\|h\|_{\tilde{X}})\cdot \|h\|_X \,.
\end{equation}
Moreover, by (\ref{rr.ass.2})
\begin{equation}\label{rr.newcomp.4}
\|F_2(F_1(u) + k) - F_2(F_1(u)) - L_2 k\|_Z
\le \rho_{2,F_1(u)}(\|k\|_{\tilde{Y}})\cdot \|k\|_Y \,.
\end{equation}
Since $F_1$ is locally Lipschitz, $\|k\|_Y \le C_1\|h\|_X$ for small
enough $\|h\|_X$.

Now let us define $\tilde{\rho}_u:\real_+\to\real_+$ by
\begin{equation}\label{rr.newcomp.5}
\tilde{\rho}_u(\lambda) = \sup \{\|F_1(u+h) - F_1(u)\|_{\tilde{Y}}:
\|h\|_{\tilde{X}} \le \lambda\} \,.
\end{equation}
By part (iv) of Assumption \ref{rr.ass}, 
$\tilde{\rho}_u(\lambda)\to 0$ as $\lambda\to 0$.
Putting together the estimates obtained so far, we get
\begin{equation}\label{rr.newcomp.6}
\begin{split}
&\|(F_2\circ F_1)(u+h) - (F_2\circ F_1)(u) - (L_2\circ L_1)h\|_Z
\\ &\qquad \qquad  \le
\big(C_1 \rho_{2,F_1(u)}(\tilde{\rho}_u(\|h\|_{\tilde{X}})) +
C \rho_{1,u}(\|h\|_{\tilde{X}})\big)\cdot \|h\|_X
\end{split}
\end{equation}
independent from the choice of $L_1$ and $L_2$, as long as $\|h\|_X$
is sufficiently small. Setting
\[
\rho_u(\lambda) =
C_1 \rho_{2,F_1(u)}(\tilde{\rho}_u(\lambda)) + C \rho_{1,u}(\lambda)
\]
we have $\rho_u(\lambda)\to 0$ as $\lambda\to 0$. Thus, it follows
from (\ref{rr.newcomp.6}) that (\ref{rr.newcomp.100}) holds.
\end{proof}

In order to obtain the refined chain rule for Bouligand derivatives,
we replace Assumption \ref{rr.ass}(iii) by
\begin{quote}
$F_1$ and $F_2$ are Bouligand differentiable in $U$ and $V$, respectively.
For every $u\in U$, $v\in V$ we have
\begin{equation}\label{rr.ass.5}
\begin{split}
\|F_1(u+h) - F_1(u) - F_1'(u;h)\|_Y &\le \rho_{1,u}(\|h\|_{\tilde{X}})\cdot \|h\|_X
\\
\|F_2(v+k) - F_2(v) - F_2'(v;k)\|_Z &\le \rho_{2,v}(\|k\|_{\tilde{Y}})\cdot \|k\|_Y
\end{split}
\end{equation}
for every $h\in X$ with $u+h\in U$ and every $k\in Y$ with $v+k\in V$.
\end{quote}
\begin{lemma}\label{nwd.chain}
If $F_1$ and $F_2$ are Hadamard differentiable at $u$ resp. $F_1(u)$, then
$F_2\circ F_1$ is Hadamard differentiable at $u$, and the chain rule
\begin{equation}\label{nwd.chain.1}
(F_2\circ F_1)'(u;h) = F_2'(F_1(u);F_1'(u;h))
\end{equation}
holds for all $h\in X$.
\hfill$\Box$
\end{lemma}
\begin{proof}
See e.g. \cite{BoS}, Proposition 2.47.
\end{proof}
\begin{proposition}[Refined Chain Rule, Bouligand Derivative]\label{rr.boucomp}\hfill\\
Let (i) - (iv) of Assumption \ref{rr.ass} hold, with (iii) replaced by (\ref{rr.ass.5}).
Then $F = F_2\circ F_1$ is Bouligand differentiable in $U$, and
\begin{equation}\label{rr.boucomp.1}
F'(u;h) = F_2'(F_1(u);F_1'(u;h)) \,.
\end{equation}
Moreover, for every $u\in U$ and $h\in X$ with $u+h\in U$
\begin{equation}\label{rr.boucomp.2}
\|F(u+h) - F(u) - F_2'(F_1(u);F_1'(u;h))\|_Z \le \rho_u(\|h\|_{\tilde{X}}) \|h\|_X
\end{equation}
for some $\rho_u:\real_+\to\real_+$ with 
$\rho_u(\delta)\downarrow 0$ as $\delta\downarrow 0$.
\end{proposition}
\begin{proof}
By Lemma \ref{nwd.chain}, $F$ is Hadamard differentiable and the
chain rule holds. It remains to show (\ref{rr.boucomp.2}) for the remainder.
Let $u\in U$, $h\in X$ with $u+h\in U$, set
$k = F_1(u+h) - F(u)$. We have
\begin{align}
&F_2(F_1(u+h)) - F_2(F_1(u)) - F_2'(F_1(u);F_1'(u;h)) =
\nonumber \\ &\qquad\qquad \qquad
\big( F_2(F_1(u)+k) - F_2(F_1(u)) - F_2'(F_1(u);k)) \big)
\label{rr.boucomp.3}
\\ &\qquad\qquad \qquad \quad
+ \big( F_2'(F_1(u);k)) - F_2'(F_1(u);F_1'(u;h)) \big)
\nonumber
\end{align}
Let $C_i$ be local Lipschitz constants for $F_i$.
The inequality
\[  
C_2\|k - F_1'(u;h))\|_X \le C_2 \rho_{1,u}(\|h\|_{\tilde{X}}) \|h\|_X
\]  
yields an estimate for the second term on the right side of
(\ref{rr.boucomp.3}); the first term is estimated by
\[  
\rho_{2,F_1(u)}(\|k\|_{\tilde{Y}})\cdot \|k\|_Y \,.
\]  
Since $\|k\|_Y \le C_1 \|h\|_X$, we argue as in the proof of
Proposition \ref{rr.newcomp} and obtain, with $\tilde{\rho}_u$
defined as in (\ref{rr.newcomp.5}),
\begin{align*}\label{rr.boucomp.7}
&\|(F_2\circ F_1)(u+h) - (F_2\circ F_1)(u) - F_2'(F_1(u);F_1'(u;h))\|_Z
\\ &\qquad \qquad  \le
\big(C_1 \rho_{2,F_1(u)}(\tilde{\rho}_u(\|h\|_{\tilde{X}})) +
C_2 \rho_{1,u}(\|h\|_{\tilde{X}})\big)\cdot \|h\|_X \,.
\end{align*}
From this, the claim readily follows.
\end{proof}

\section{The scalar play and stop operators}

The original construction of the play and the stop operators
in \cite{KDE70} is based on piecewise monotone input functions.
A continuous function $u:[a,b]\to\real$ is called \textbf{piecewise monotone},
if the restriction of $u$ to each interval $[t_i,t_{i+1}]$ of a suitably
chosen partition $\Delta_{pm} = \{t_i\}$, $a = t_0 < t_1 < \dots < t_N = b$,
called a \textbf{monotonicity partition} of $u$,
is either nondecreasing or nonincreasing. By $C_{pm}[a,b]$ we denote
the space of all such functions.

For arbitrary $r\ge 0$, the play operator $\playr$ and the stop operator
$\stopr$ are constructed as follows.
(For more details, we refer to Section 2.3 of \cite{BS}.)
Given a function $u\in C_{pm}[a,b]$ and an initial value $z_0\in [-r,r]$,
we define functions $w,z:[a,b]\to\real$
successively on the intervals $[t_i,t_{i+1}]$, $0 \le i < N$, of a monotonicity
partition $\Delta_{pm}$ of $u$ by
\begin{equation}\label{po.0}
z(a) = \pi_r(z_0) := \max\{-r,\min\{r,z_0\}\} \,,\quad
w(a) = u(a) - z(a) \,,
\end{equation}
and
\begin{equation}\label{po.1}
\begin{aligned}
w(t) &= \max\{u(t) - r \,,\,\min\{u(t) + r , w(t_i)\}\} \,,
\\
z(t) &= v(t) - w(t) \,,
\end{aligned}
\qquad t_i < t \le t_{i+1} \,.
\end{equation}
In this manner, we obtain operators
\[
w = \playr[u;z_0] \,,\quad  z = \stopr[u;z_0] \,,\quad
\playr,\stopr: C_{pm}[a,b]\times\real\to C_{pm}[a,b] \,.
\]
By construction,
\begin{equation}\label{po.01}
u = w + z = \playr[u;z_0] + \stopr[u;z_0] \,.
\end{equation}
The play operator satisfies
\begin{equation}\label{po.4}
\|\playr[u;z_0] - \playr[v;y_0]\| \le \max\{\|u - v\| \,,\, |z_0 - y_0|\}
\end{equation}
for all $u,v\in C_{pm}[a,b]$ and all $z_0,y_0\in\real$.
Therefore, $\playr$ and $\stopr$ can be uniquely extended to Lipschitz continuous
operators
\[
\playr,\stopr: C[a,b]\times\real\to C[a,b]
\]
which satisfy (\ref{po.4}) for all $u,v\in C[a,b]$ and all $z_0,y_0\in\real$.

In \cite{BK}, Hadamard derivatives of $\playr$ and of $\stopr$ have been
obtained. We recall some of the terminology used there, as it is also
relevant for the present paper.

Let $(u,z_0)\in C[a,b]\times\real$ be given, let $w = \playr[u;z_0]$,
$z = \stopr[u;z_0]$ with $r > 0$. (For $r=0$, $\playr[u;z_0] = u$.)
The trajectories $\{(u(t),w(t)): t\in [a,b]\}$
lie within the subset $A = \{|u-w| \le r\}$ of the plane $\real^2$ 
bounded by the straight lines $u-w = \pm r$.
They consist of parts which belong to the interior, the right or the left
boundary of $A$.
Correspondingly, the time interval $[a,b]$ decomposes into the three disjoint
sets
\begin{align*}
I_0 &= \{t\in [a,b]: |u(t) - w(t)| = |z(t)| < r\} \,,
\\
I_{\partial +} &= \{t\in [a,b]: u(t) - w(t) = z(t) = r\} \,,
\\
I_{\partial -} &= \{t\in [a,b]: u(t) - w(t) = z(t) = -r \} \,.
\end{align*}
The set $I_0$ is an open subset of $[a,b]$, the sets $I_{\partial\pm}$
are compact. As $I_{\partial +}$ and $I_{\partial -}$ are disjoint,
\begin{equation}\label{po.9}
\delta_I :=
\min\{|\tau-\sigma|: \tau\in I_{\partial +}\,,\,\sigma\in I_{\partial -}\} > 0 \,.
\end{equation}
Because of this, there exists a finite partition $\Delta(u,z_0) = \{t_k\}$ of $[a,b]$
such that on each partition interval $I_k = [t_{k-1},t_k]$ we have
$z(t) > -r$ for all $t\in I_k$ or $z(t) < r$ for all $t\in I_k$,
or both.
In the former case, $I_k$ is called a \textbf{plus interval};
on $I_k$ the trajectory stays away from the left boundary of $A$,
and $I_k \subset I_0\cup I_{\partial +}$.
In the latter case, $I_k$ is called a \textbf{minus interval};
the trajectory stays away from the right boundary of $A$,
and $I_k \subset I_0\cup I_{\partial -}$.
Note that if $I_k\subset I_0$, then $I_k$ is a plus as well as a minus interval.

It has been proved in \cite{BK}, Lemma 5.1, that on such intervals the play
operator behaves like an cumulated maximum resp. minimum. More precisely,
on a plus interval $I_k$,
\begin{equation}\label{po.5}
w(t) = \playr[u;z_0](t) = \max \{ w(t_{k-1})\,,\,\max_{s\in [t_{k-1},t]} (u(s) - r)\}
\end{equation}
holds, no matter whether $u$ is monotone on $I_k$ or not.
On a minus interval,
\begin{equation}\label{po.6}
w(t) = \playr[u;z_0](t) = \min \{ w(t_{k-1})\,,\,\min_{s\in [t_{k-1},t]} (u(s) + r)\} \,.
\end{equation}
In particular, $w(t) = w(t_{k-1})$ if $I_k \subset I_0$.

Due to (\ref{po.9}) and the continuity of $\playr$, in this manner the
play and the stop operator can locally be represented by a finite
composition of operators arising from the cumulated maximum resp.
minimum. The following result has been proved in \cite{BK},
Lemma 5.2.
\begin{proposition}\label{po.locrep}
For every $(u,z_0)\in C[a,b]\times\real$ there exists a partition
$\Delta(u,z_0) = \{t_k\}_{0\le k\le N}$ of $[a,b]$ and a $\delta > 0$ such
that every partition interval $[t_{k-1},t_k]$ of $\Delta$ is a plus interval
for all $(v,y_0)\in U_\delta\times\real$, or it is a minus interval
for all $(v,y_0)\in U_\delta\times\real$. Here,
\begin{equation}\label{po.locrep.1}
U_\delta := \{(v,y_0):
\text{$\|v - u\|_\infty < \delta$, $|y_0 - z_0| < \delta$, $v\in C[a,b]$,
$y_0\in\real$}\}
\end{equation}
is the $\delta$-neighbourhood of $(u,z_0)$ w.r.t the maximum norm.
\hfill$\Box$
\end{proposition}
As a consequence, invoking the chain rule for Hadamard derivatives,
it has been proved in \cite{BK} that $\playr$ and $\stopr$ are
Hadamard differentiable on $C[a,b]\times\real$, if $L^q(a,b)$ with
$q < \infty$ is chosen as the range space.
\section{Newton derivative of the play and the stop}
We want to use the approach outlined in the previous section
in order to obtain a Newton derivative of $\playr$,
based on the Newton derivative of the cumulated maximum.
\par
We want to construct the Newton derivative such
that its dependence upon $(u,z_0)$ becomes measurable in a suitable
manner; for this, the local representation of the play
obtained from Proposition \ref{po.locrep} seems to be of very limited value.
Instead, we employ properties of the set-valued mappings involved when
constructing above the Newton derivative of the cumulated maximum.
To this purpose, we turn around the approach of Proposition \ref{po.locrep}.
Instead of finding a suitable partition $\Delta$ for a given $(u,z_0)$, 
for a given partition $\Delta$ we consider sets of $(u,z_0)$ for which
the play can be ``decomposed'' by $\Delta$.

Throughout the following, the space $X$ stands for $C^{0,\alpha}[a,b]$ or $W^{1,p}(a,b)$.

Let $\Delta = \{t_k\}$ be a partition of $[a,b]$, $a = t_0 < \dots < t_N = b$
for some $N\in\nat$. We set
\[
I_k = [t_{k-1},t_k] \,,\quad
|\Delta| = \max_{1\le k\le N} |I_k| = \max_{1\le k\le N} (t_k - t_{k-1}) \,.
\]
We define
\begin{equation}\label{po.part.1}
\begin{split}
C^\Delta &= \{u: u\in C[a,b],\,z_0\in\real,\,
\osc_{I_k}(u) < r \text{ for all $k$}\} \,,
\\ X^\Delta &= X\cap C^\Delta \,,
\\  Z^\Delta &= C^\Delta \times \real =
\{(u,z_0): u\in C^\Delta,\,z_0\in\real \} \,.
\end{split}
\end{equation}
The sets $C^\Delta$, $X^\Delta$ and $Z^\Delta$ are open subsets of $C[a,b]$,
$X$ and $C[a,b]\times\real$, respectively.


\textbf{The dynamics on an interval for small input oscillation.}

It turns out below in Proposition \ref{po.decp} that an interval
$I\subset [a,b]$ is a plus or a minus interval for the play if the 
oscillation of $u$ on $I$ is less than $r$. This and some other
auxiliary results are developed up to Proposition \ref{po.decbn}.

Let $I = [t_*,t^*]\subset [a,b]$, $u\in C(I)$. We denote the cumulated
maximum of $u$ on $I$ and the sets where it is attained by
\begin{equation}\label{po.dec.1}
\begin{split}
(F^Iu)(t) = \max_{s\in I,s\le t} u(s) \,,\quad t\in I \,, \qquad \quad
\\
M^I(u,t) = \{s: s\in I, s\le t,u(s) = (F^Iu)(t)\} \,.
\end{split}
\end{equation}
As above, $F^I:C(I)\to C(I)$, $F^I u$ is nondecreasing and
$F^I(u+c) = F^I(u) + c$ if $c$ is a constant. Moreover,
\begin{gather}\label{po.dec.2}
\osc_I(F^Iu) = u(t^*) - u(t_*) \le \osc_I(u) \,,
\\ \label{po.dec.3}
0 \le F^I(u) - u \le \osc_I u  \quad \text{on $I$,}
\end{gather}
and consequently
\begin{equation}\label{po.dec.4}
\osc_I(F^I(u) - u) \le \osc_I u \,.
\end{equation}
The cumulated minimum of $u$ on $I$ can be written as
\begin{equation}\label{po.dec.42}
\min_{s\in I,s\le t} u(s) = - (F^I(-u))(t) \,,\quad t\in I \,.
\end{equation}
The corresponding sets of minima are given by $M(-u,t)$.

For $u\in C(I)$, $p\in\real$ and $r>0$ we define the functions
(here and in the following, the max and the min are taken pointwise in $t$)
\begin{equation}\label{po.dec.5}
\begin{split}
w_+ = \max\{p, F^I(u-r)\} \,,\quad z_+ = u - w_+ \,,\quad
\\
w_- = \min\{p, - F^I(-u-r)\} \,,\quad z_- = u - w_- \,.
\end{split}
\end{equation}
This corresponds to the operations in (\ref{po.5}) and (\ref{po.6}).
We have $w_+,w_-,z_+,z_-\in C(I)$. Obviously $w_-\le w_+$, $z_+\le z_-$.

Since $p\le w_+ = u - z_+$ and $p\ge w_- = u - z_-$, we have
\begin{equation}\label{po.dec.6}
z_+ \le u - p \le z_- \,.
\end{equation}
\begin{lemma}\label{po.dec.zrand}
Let $u\in C(I)$, $p\in\real$, $r > 0$. \hfill\\
(i) We have $z_+\le r$ on $I$. If $z_+(t) = r$ for some $t\in I$, then
$u(t) = (F^I u)(t) \ge p + r$. \hfill\\
(ii) We have $z_-\ge -r$ on $I$. If $z_-(t) = -r$ for some $t\in I$, then
$u(t) = -(F^I (-u))(t) \le p - r$.
\end{lemma}
\begin{proof} To obtain (i), we use the estimate
\begin{align*}
z_+ &= u - w_+ = u - \max\{p, F^I(u-r)\}
= u - p - \max\{0, F^I(u-r-p)\}
\\
&\le u - p - F^I(u-r-p) = u - F^I(u) + r \le r \,.
\end{align*}
If $z_+(t) = r$, equality holds everywhere, so $u(t) = (F^Iu)(t)$
and $F^I(u-r-p)(t) \ge 0$. The proof of (ii) is analogous.
\end{proof}

We consider inputs in $C(I)$ whose oscillation is smaller than $r$.
\begin{equation}\label{po.dec.7}
\begin{split}
Z^I &= \{(u,p): u\in C(I),p\in\real,\osc_I u < r\}
\\
Z_+^I &= \{(u,p): u\in C(I),p\in\real,\osc_I u < r, z_+ > -r \text{ on $I$}\}
\\
Z_-^I &= \{(u,p): u\in C(I),p\in\real,\osc_I u < r, z_- < r \text{ on $I$}\}
\end{split}
\end{equation}
The sets $Z_+^I$ and $Z_-^I$ are open subsets of $Z^I$ in $C(I)\times\real$;
we will see that they correspond to plus and minus intervals for the play.
\begin{lemma}\label{po.dec.wpr}
\hfill\\ (i)
If $(u,p)\in Z_-^I$ then $F^I(u-r-p) < 0$ and $w_+ = p$ on $I$.
\hfill\\ (ii)
If $(u,p)\in Z_+^I$ then $F^I(-u-r+p) < 0$ and $w_- = p$ on $I$.
\hfill\\ (iii)
If $(u,p)\in Z_-^I\cap Z_+^I$ then $w_+ = w_- = p$ and $z_+ = z_- = u-p$ on $I$.
\end{lemma}
\begin{proof}
If $(u,p)\in Z_-^I$ then $u-p-r \le z_- -r < 0$ by (\ref{po.dec.6}),
so $F^I(u-r) - p < 0$, so $w_+ = p$.
If $(u,p)\in Z_+^I$ then $-u+p-r \le -z_+ -r < 0$ by (\ref{po.dec.6}),
so $F^I(-u-r) + p < 0$, so $w_- = p$.
\end{proof}
\begin{lemma}\label{po.dec.zpr}
Let $u\in C(I)$, $\osc_I(u) < r$, $p\in\real$. Then
\begin{equation}\label{po.dec.zpr.1}
\min\{u-p,0\} \le z_+ \le z_- \le \max\{u-p,0\} \,.
\end{equation}
\end{lemma}
\begin{proof}
We have
\[
- z_+ = w_+ - u = \max\{p-u, F^I(u) - r - u\} \le \max\{ p - u, 0\} \,,
\]
since $F^I u - u \le \osc_I u < r$ by (\ref{po.dec.3}). Analogously,
\[
- z_- = w_- - u = \min\{p-u, - F^I(-u - r) - u\} \ge \min\{ p - u, 0\} \,,
\]
since $ F^I (-u) - (-u) \le \osc_I (-u) < r$ by (\ref{po.dec.3}).
\end{proof}
\begin{lemma}\label{po.dec.zalles}
We have $Z^I = Z_+^I\cup Z_-^I$.
\end{lemma}
\begin{proof}
Let $(u,p)\in Z^I$,
assume that $(u,p)\notin Z_+^I$. Then $z_+(t) \le -r$ for some $t\in I$.
By (\ref{po.dec.zpr.1}), $u(t) - p \le -r$. As $\osc_I(u) < r$, we have
$u - p \le 0$ on $I$. By (\ref{po.dec.zpr.1}), $z_- \le 0$ on $I$,
so $(u,p)\in Z_-^I$.
\end{proof}
We define $P_+^I:Z_+^I \to C(I)$ and $P_-^I:Z_-^I \to C(I)$ by
\begin{equation}\label{po.dec.10}
\begin{split}
P_+^I(u,p) &= p + \max\{0, F^I(u-r-p)\} \,,
\\
P_-^I(u,p) &= p - \max\{0, F^I(-u-r+p)\} \,.
\end{split}
\end{equation}
Therefore, in view of (\ref{po.dec.5}),
\begin{align}\label{po.dec.101}
&u - P^I_+(u,p) = u - w_+ = z_+ > -r \quad \text{on $I$}
\quad \Leftrightarrow \quad (u,p)\in Z^I_+,
\\ \label{po.dec.102}
&u - P^I_-(u,p) = u - w_- = z_- < r \;\;\; \quad   \text{on $I$}
\quad \Leftrightarrow \quad (u,p)\in Z^I_- \,.
\end{align}
%
%
On $Z_+^I\cap Z_-^I$ both expressions simplify to $P_\pm^I(u,p) = p$
by Lemma \ref{po.dec.wpr}. Therefore,
\begin{equation}\label{po.dec.11}
P^I(u,p) = P_\pm^I(u,p) \,,\quad \text{if $(u,p)\in Z_\pm^I$}
\end{equation}
yields a well-defined mapping $P^I: Z^I\to C(I)$.
%
%
%

The next result states that for $u\in C^\Delta$ the intervals $I_k$
yield a decomposition of the play operator. This is the analogue of
Proposition \ref{po.locrep}.

\begin{proposition}\label{po.decp}
Let $u\in C^\Delta$ and $z_0\in\real$, set $p = \playr[u;z_0](t_{k-1})$,
$k\ge 1$. Then \hfill\\
\begin{equation}\label{po.decp.2}
w(t) = \playr[u;z_0](t) = P^{I_k}(u,p)(t) \,,\quad \text{for all $t\in I_k$.}
\end{equation}
Moreover,
\begin{equation}\label{po.decp.20}
\begin{split}
\text{$I_k$ is a plus interval} \quad \Leftrightarrow \quad
(u,p)\in Z_+^{I_k} \,,
\\
\text{$I_k$ is a minus interval} \quad \Leftrightarrow \quad
(u,p)\in Z_-^{I_k} \,.
\end{split}
\end{equation}
\end{proposition}
\begin{proof}
Let $\{u_n\}$ be a sequence in $C[a,b]$ such that the functions $u_n$
coincide with $u$ on $[0,t_{k-1}]$, are piecewise linear on $I_k$
and satisfy $u_n\to u$ uniformly.
For $n$ large enough we have $u_n\in C^\Delta$, so $z_n > -r$ on $I_k$
if $(u,p)\in Z_+^{I_k}$ and $z_n < r$ if $(u,p)\in Z_-^{I_k}$. It follows
that $w_n = \playr[u_n;z_0] = P^{I_k}(u_n,p)$ on $I_k$, by the
definition of the play on $C_{pm}(I_k)$, see (\ref{po.1}).
Passing to the limit $n\to\infty$ yields (\ref{po.decp.2}).
To prove the first equivalence in (\ref{po.decp.20}), let $I_k$ be a plus
interval. On $I_k$ we then have $u-w > -r$ and, by (\ref{po.5}),
$w = P_+^{I_k}(u,p)$, so $(u,p) \in Z_+^{I_k}$ by (\ref{po.dec.101}).
Conversely, if $(u,p) \in Z_+^{I_k}$, we have
$u - P^{I_k}(u,p) = u - P_+^{I_k}(u,p) > -r$ on $I_k$ by (\ref{po.dec.101}),
so $u - w > -r$ by (\ref{po.decp.2}). The proof of the second
equivalence is analogous.
\end{proof}
We specify some properties of points in the ``boundary sets'' $I_{\partial\pm}$.
\begin{proposition}\label{po.decbp}
Let $u\in C^\Delta$ and $z_0\in\real$, set $p = \playr[u;z_0](t_{k-1})$,
$k\ge 1$ and $w = P^{I_k}(u,p)$. Then the following holds.
\hfill\\
(i) Let $t\in I_{\partial +} \cap I_k$. Then 
$(u,p)\in Z^{I_k}_+$ and
\begin{equation}\label{po.decp.3}
u(t) = (F^{I_k}u)(t) \ge p + r \,.
\end{equation}
(ii) Let $\tau\in I_{\partial +} \cap I_k$, $\tau \le t \le t_k$. Then
$M^{I_k}(u,t)\cap [\tau,t] \subset I_{\partial +}$.
\hfill\\
(iii) Let $(u,p)\in Z_+^{I_k}$, let $t\in I_k$ with $(F^{I_k}u)(t) \ge p + r$.
Then there exists $\tau\in [t_{k-1},t]$ with $\tau\in I_{\partial +}$.
\hfill\\
(iv) Let $(u,p)\in Z_+^{I_k}$, let $s,t\in I_k$ with $s < t$ and $w(s) < w(t)$.
Then $M^{I_k}(u,t) \subset (s,t]$.
\end{proposition}

\begin{proof}
(i) Let $t\in I_{\partial +} \cap I_k$, so $u(t) - w(t) = r$.
We have $(u,p)\in Z^{I_k}_+$ since otherwise by (\ref{po.decp.2}),
(\ref{po.dec.11}) and (\ref{po.dec.102})
\[
w(t) = P^{I_k}(u,p)(t) = P^{I_k}_-(u,p)(t) = u(t) - z_-(t) > u(t) - r \,,
\]
a contradiction. 
Since $z_+(t) = r$, the remaining assertions are a direct
consequence of Lemma \ref{po.dec.zrand}.
\hfill\\
(ii) Let $s\in M^{I_k}(u,t)$ with $s\ge\tau$. As $F^{I_k}(u-p-r)$ is nondecreasing
and $F^{I_k}(u-p-r)(\tau)\ge 0$ by (\ref{po.decp.3}), we have
\begin{align*}
w(s) - p &= \max \{0, F^{I_k}(u-p-r)(s)\} = F^{I_k}(u-p-r)(s)
\\ &\le
F^{I_k}(u-p-r)(t) = u(s) - p - r \,.
\end{align*}
As $|u(s) - w(s)| \le r$ it follows that $u(s) - w(s) = r$ and therefore
$s\in I_{\partial +}$.
\hfill\\
(iii) As
$F^{I_k}(u-p-r)(t_{k-1}) = u(t_{k-1}) - w(t_{k-1}) -r \le 0$, we find
$\sigma\in [t_{k-1},t]$ with $F^{I_k}(u-p-r)(\sigma) = 0$.
As $F^{I_k}(u-p-r)$ is nondecreasing,
we have $F^{I_k}(u-p-r) \le 0$ on $[t_{k-1},\sigma]$ and therefore
\[
w = P_+^{I_k}(u,p) = p = p + (F^{I_k}(u-p-r))(\sigma) = (F^{I_k}u)(\sigma) - r
\]
on $[t_{k-1},\sigma]$.
We choose $\tau\in [t_{k-1},\sigma]$ with $u(\tau) = (F^{I_k}u)(\sigma)$.
Then $w(\tau) = u(\tau) - r$, so $\tau \in I_{\partial +}$.
\hfill\\
(iv) By (\ref{po.decp.2}),
\[
\max \{p, F^{I_k}(u-r)(s)\} = w(s) < w(t) = \max \{p, F^{I_k}(u-r)(t)\} ,
\]
so $(F^{I_k}u)(s) < (F^{I_k}u)(t)$ and therefore $M^{I_k}(u,t) \subset (s,t]$.
\end{proof}
For minus intervals, the corresponding results read as follows. Their proofs
are analogous to those of Proposition \ref{po.decbp}.
\begin{proposition}\label{po.decbn}
Let $u\in C^\Delta$ and $z_0\in\real$, set $p = \playr[u;z_0](t_{k-1})$,
$k\ge 1$ and $w = P^{I_k}(u,p)$. Then the following holds.
\hfill\\
(i) Let $t\in I_{\partial -} \cap I_k$. Then 
$(u,p)\in Z^{I_k}_-$ and
\begin{equation}\label{po.decbn.1}
u(t) = -(F^{I_k}(-u))(t) \le p - r \,.
\end{equation}
(ii) Let $\tau\in I_{\partial -} \cap I_k$, $\tau \le t \le t_k$. Then
$M^{I_k}(u,t)\cap [\tau,t] \subset I_{\partial -}$.
\hfill\\
(iii) Let $(u,p)\in Z_-^{I_k}$, let $t\in I_k$ with $-(F^{I_k}(-u))(t) \le p - r$.
Then there exists $\tau\in [t_{k-1},t]$ with $\tau\in I_{\partial -}$.
\hfill\\
(iv) Let $(u,p)\in Z_-^{I_k}$, let $s,t\in I_k$ with $s < t$ and $w(s) > w(t)$.
Then $M^{I_k}(-u,t) \subset (s,t]$.
\hfill$\Box$
\end{proposition}
\textbf{A Newton derivative on an interval of small input oscillation.}
We want to obtain a Newton derivative for $P^I:Z^I\cap (X\times\real) \to L^q(I)$,
where $I = [t_*,t^*]\subset [a,b]$.
The mapping $P_+^I$ decomposes into
\begin{equation}\label{po.pi.1}
P_+^I(u,p) = p + F_{pp}(\tilde{F}^I(u,p)) \,.
\end{equation}
Here, $\tilde{F}^I: Z^I\cap (X\times\real)\to C(I)$ is defined as
\begin{equation}\label{po.pi.2}
\tilde{F}^I(u,p) = F^I(u-p-r) \,,
\end{equation}
and $F_{pp}$ denotes the positive part mapping
\begin{equation}\label{po.pi.3}
(F_{pp}u)(t) = \max\{0,u(t)\} \,.
\end{equation}
We first analyze the mapping $\tilde{F}^I$.
We expect to obtain a Newton derivative $\tilde{G}^I$ of $\tilde{F}^I$
if we choose elements $L\in \tilde{G}^I(u,p)$ of the form
\begin{equation}\label{po.pi.4}
(L(h,\eta))(t) =  \langle \mu^I(u,t) \,,\, h -\eta\rangle \,,
\end{equation}
where $\mu^I(u,t)$ are probability measures arising from the derivative
of the cumulated maximum on $I$.
We define $T: C(I)\times\real\to C(I)$ by $T(u,p) = u - p$
and consider the set-valued mapping
\begin{equation}\label{po.pi.5}
\begin{split}
\tilde{\Psi}^I:Z^I\times I \toto (C(I)\times\real)^* \quad \quad
\\
\tilde{\Psi}^I(u,p,t) = \Phi^I(T(u,p)-r,t)\circ T\,.
\end{split}
\end{equation}
$\Phi^I$ is the mapping defined in (\ref{mam.phit.1}) with
$[a,b]$ and $M$ replaced with $I$ and $M^I$ from (\ref{po.dec.1}).
We compute $\tilde{\Psi}^I$, using (\ref{mam.phit.3}) and (\ref{mam.phit.4}),
\begin{equation}\label{po.pi.6}
\begin{split}
\tilde{\Psi}^I(u,p,t) &= \Phi^I(u,t)\circ T
= \Phi^I(u,t)\circ (\pi_1 - j \circ \pi_2)
\\ &= \Phi^I(u,t)\circ \pi_1 - \pi_2 \,.
\end{split}
\end{equation}
Here, $\pi_1,\pi_2:C(I)\times\real\to\real$ denote the projections
$\pi_1(h,\eta) = h$ and $\pi_2(h,\eta) = \eta$,
and $j$ maps real numbers to the corresponding constant functions.
We see that $\tilde{\Psi}^I(u,p,t)$ actually does not depend on $p$.

Let $S_\Phi^I$ be the set of all measurable selectors of $\Phi^I$.
\begin{proposition}\label{po.find}
The mapping $\tilde{\Psi}^I$ is usc and has $w^*$-compact values. Moreover,
\begin{equation}\label{po.find.1}
\tilde{S}_\Psi^I =
\{ \tilde{\mu}: \tilde{\mu}(u,p,t) = \mu^I(u,t)\circ\pi_1 - \pi_2,
\mu^I\in S_{\Phi}^I \}
\end{equation}
is a set of measurable selectors of $\tilde{\Psi}^I$.
For $\tilde{F}^I: Z^I\cap (X\times\real)\to L^{\tilde{q}}(I)$
with $\tilde{q} < \infty$,
a Newton derivative $\tilde{G}^I$ is given by
\begin{equation}\label{po.find.2}
\begin{split}
\tilde{G}^I: Z^I\cap (X\times\real)\toto \mathcal{L}(X\times\real,L^{\tilde{q}}(I))
\qquad
\\
\tilde{G}^I(u,p) = \{L: \text{$L$ has the form (\ref{po.pi.4})
with $\mu^I\in S_{\Phi}^I$} \} \,.
\end{split}
\end{equation}
The elements $L$ of $\tilde{G}^I(u,p)$ satisfy
\begin{equation}\label{po.find.21}
\|L(h,\eta)\|_{\infty,t} \le \|h\|_{\infty,t} + |\eta|
\end{equation}
for all $h\in C(I)$, $\eta\in\real$, $t\in I$. 
Moreover, the remainder estimate
\begin{equation}\label{po.find.22}
\begin{split}
&\sup_{L\in \tilde{G}^I(u+p,h+\eta)}
\|\tilde{F}^I(u+h,p+\eta) - \tilde{F}^I(u,p) - L(h,\eta)\|_{L^{\tilde{q}}(I)} 
\\&\qquad \qquad \le 
\rho_{(u,p)}(\|h\|_\infty + |\eta|) \|(h,\eta)\|_{X\times\real}
\end{split}
\end{equation}
holds. 
The remainder term $\rho_{(u,p)}$ satisfies $\rho_{(u,p)}(\delta)\downarrow 0$
as $\delta\downarrow 0$ and is uniformly bounded in $(u,p)$.
\end{proposition}

\begin{proof}
Let $A: C(I)^* \to (C(I)\times\real)^*$, $A(\nu) = \nu\circ \pi_1 - \pi_2$.
Then $A$ is linear and $w^*$-$w^*$-continuous, and
$\tilde{\Psi}^I = A \circ \tilde{\Phi}^I$.
Since $\Phi^I$ is usc according to Proposition \ref{mam.phitusc} and has
$w^*$-compact values, using Lemma \ref{svm.comp} we see that
the same is true for $\tilde{\Psi}^I$.


The elements of $\tilde{S}_\Psi^I$
are measurable as compositions of measurable functions. As $F^I$ has a
Newton derivative given by Proposition \ref{am.newder} and
$\tilde{F}^I(u,p) = F^I(u-p-r)$, setting $\tilde{X} = \tilde{Y} = C(I)$ 
and $Z = L^{\tilde{q}}(I)$ we 
check that the assumptions of the refined chain rule, Proposition \ref{rr.newcomp},
are satisfied. Therefore,
$\tilde{G}^I$ is a Newton derivative of $\tilde{F}^I$ and (\ref{po.find.22}) holds.
(\ref{po.find.21}) is a consequence of (\ref{po.pi.4}) and (\ref{mam.l.1}).
Since $\tilde{F}$ is globally Lipschitz w.r.t. the maximum norm,
together with (\ref{po.find.21}) the final assertion follows.
\end{proof}

Since $\Phi^I(v,t_*) = \{\delta_{t_*}\}$ for all $v$,
by (\ref{po.find.1}) we have for all $\tilde{\mu}\in \tilde{S}^I_\Psi$
\begin{equation}\label{po.find.3}
\langle \tilde{\mu}(u,p,t_*), (h,\eta) \rangle = h(t_*) - \eta
\end{equation}
for all $(u,p)\in Z^I$ and all $(h,\eta)\in C(I)\times I$.

For functions $u:I\to\real$, we consider the positive part mapping
$F_{pp}$ defined by
\begin{equation}\label{po.pp.1}
(F_{pp}u)(t) = \max\{0,u(t)\} \,,
\end{equation}
which maps $L^q(I)$ as well as $C(I)$ into itself. Let $H:\real\toto\real$
be the set-valued Heaviside function
\begin{equation}\label{po.pp.2}
H(x) = \begin{cases} 0 \,, & x < 0 \,, \\ [0,1] \,,& x = 0 \,,\\ 1 \,,& x > 0 \,.
\end{cases}
\end{equation}
The mapping $H$ is usc. By
\begin{equation}\label{po.pp.21}
S_H = \{\lambda_H: \text{$\lambda_H$ selector of $H$, $\lambda_H(0)\in\mathbb{Q}$}\}
\end{equation}
we define a countable family of measurable selectors of $H$ whose values are
dense in the range of $H$. We then define
\begin{equation}\label{po.pp.3}
\begin{split}
G_{pp}(u) : L^{\tilde{q}}(I)\toto L(L^{\tilde{q}}(I), L^q(I)) \qquad \quad
\\
G_{pp}(u) = \{L: L(h) = (\lambda_H\circ u)\cdot h ,\, \lambda_H\in S_H \} \,.
\end{split}
\end{equation}
\begin{lemma}\label{po.ppnd}
The mapping $G_{pp}$ is a Newton derivative of
$F_{pp}: L^{\tilde{q}}(I)\to L^q(I)$ for $1\le q < \tilde{q}\le \infty$.
\end{lemma}

\begin{proof}
This is a well-known result, see Proposition 3.49 in \cite{Ulb11} or
Example 8.14 in \cite{IK}.
\end{proof}

With the composition $F_{pp}\circ \tilde{F}^I$ we associate the set-valued mapping
\begin{equation}\label{po.pp.5}
\begin{split}
\tilde{\Psi}^I_{pp}: C(I)\times\real\times I \toto \real
\\
\tilde{\Psi}^I_{pp}(u,p,t) = H(\tilde{F}^I(u,p)(t)) \,.
\end{split}
\end{equation}
By the definition of $H$,
\begin{equation}\label{po.pp.6}
\tilde{\Psi}^I_{pp}(u,p,t) = \{0\} \,,\quad \text{if $\tilde{F}^I(u,p)(t) < 0$.}
\end{equation}
\begin{lemma}\label{po.ppset}
The mapping $\tilde{\Psi}^I_{pp}$ defined in (\ref{po.pp.6}) is usc and has compact
values. A set of measurable selectors is given by
\begin{equation}\label{po.pp.7}
\tilde{S}^I_{pp} = \{\tilde{\lambda}:
\tilde{\lambda}(u,p,t) = \lambda_H(\tilde{F}^I(u,p)(t)),\, \lambda_H\in S_H\} \,.
\end{equation}
\end{lemma}

\begin{proof}
The mapping $\tilde{F}^I: C(I)\times\real\to C(I)$ as well as the mapping
$(v,t)\mapsto v(t)$ are continuous on $C(I)\times\real$, and $H$ is usc and
has compact values.
Therefore, $\tilde{\Psi}^I_{pp}$ is usc by Lemma \ref{svm.comp}, has compact values,
and the elements of $\tilde{S}^I_{pp}$ are measurable functions.
\end{proof}
We have now all ingredients to define a Newton derivative $G_+^I$ of
$P_+^I$. Its elements $L_+^I\in G_+^I(u,p)$ are expected to have the form
\begin{equation}\label{po.ndip.1}
L_+^I(h,\eta)(t) = \eta + \lambda_H(\tilde{F}^I(u,p)(t))
\cdot\langle \mu^I(u,t) , h-\eta \rangle
\end{equation}
with functions $\lambda_H\in S_H$ and measures $\mu^I\in S_\Phi^I$.
The associated set-valued mapping is given by
\begin{equation}\label{po.ndip.2}
\begin{split}
\Psi_+^I: C(I)\times\real\times I \toto (C(I)\times\real)^*
\qquad
\\
\Psi_+^I(u,p,t) = \pi_2 + \tilde{\Psi}^I_{pp}(u,p,t) \cdot \tilde{\Psi}^I(u,p,t)\,,
\end{split}
\end{equation}
where $\pi_2$ denotes the projection $\pi_2:C(I)\times\real\to\real$,
$\pi_2(h,\eta) = \eta$.
(The elementwise multiplication $\tilde{\Psi}^I_{pp}\cdot\tilde{\Psi}^I$ makes
sense since $\tilde{\Psi}^I_{pp}$ takes values in $\real$.)
We define
\begin{equation}\label{po.ndip.3}
S_+^I = \{\tilde{\nu}: \tilde{\nu}(u,p,t) =
\pi_2 + \tilde{\lambda}(u,p,t)\tilde{\mu}(u,p,t),\,
\tilde{\mu}\in \tilde{S}_\Psi^I,\,\tilde{\lambda}\in\tilde{S}^I_{pp}\} \,.
\end{equation}
\begin{proposition}\label{po.ndip}
The mapping $\Psi_+^I$ in (\ref{po.ndip.2}) is usc and has $w^*$-compact
values.
The set $S_+^I$ given in (\ref{po.ndip.3}) consists of measurable
selectors of $\Psi_+^I$.
The mapping
\begin{equation}\label{po.ndip.4}
\begin{split}
G_+^I: Z_+^I\cap (X\times\real) \toto \mathcal{L}(X\times\real,L^q(I))
\qquad \qquad \qquad
\\
G_+^I(u,p) = \{L_+^I: \text{$L_+^I$ given by (\ref{po.ndip.1}) with
$\lambda_H\in S_H$, $\mu^I\in S_\Phi^I$}\}
\end{split}
\end{equation}
is a Newton derivative of $P_+^I: Z_+^I\cap (X\times\real) \to L^q(I)$
for every $q < \infty$.
The elements $L_+^I$ of $G_+^I(u,p)$ satisfy, 
for all $(u,p)\in Z_+^I\cap (X\times\real)$,
\begin{equation}\label{po.ndip.5}
\|L_+^I(h,\eta)\|_{\infty,t} \le \max\{\|h\|_{\infty,t},|\eta|\}
\end{equation}
for all $h\in C(I)$, $\eta\in\real$. 
Moreover, for all such $(u,p)$ the remainder estimate
\begin{equation}\label{po.ndip.6}
\begin{split}
&\sup_{L_+^I\in G_+^I(u+h,p+\eta)}
\|P_+^I(u+h,p+\eta) - P_+^I(u,p) - L_+^I(h,\eta)\|_{L^q(I)} 
\\&\qquad \qquad  \le 
\rho_{(u,p)}(\|h\|_\infty + |\eta|) \|(h,\eta)\|_{X\times\real}
\end{split}
\end{equation}
holds for all $h\in X$ with $u+h\in Z_+^I$ and all $\eta\in\real$. 
The remainder term $\rho_{(u,p)}$ satisfies $\rho_{(u,p)}(\delta\downarrow 0$
as $\delta\downarrow 0$ and is uniformly bounded in $(u,p)$.
\end{proposition}

\begin{proof}
Due to Proposition \ref{po.find} and Lemma \ref{po.ppset},
we may apply Proposition \ref{svm.psicomp} with
$F_1(u,p,t) = \tilde{F}^I(u,p)(t)$, $\Psi_1 = \tilde{\Psi}^I$ and
$\Psi_2(u,p,t) = \pi_2 + \tilde{\Psi}^I_{pp}(u,p,t)\pi_3$, that is,
the elements of $\Psi_2$ have the form
\[
L_2(h,\eta,y) = \eta + L_{pp}^t\cdot y \,,\quad L_{pp}^t\in \tilde{\Psi}^I_{pp} \,.
\]
This shows that $\Psi_+^I$ is usc and has $w^*$-compact values.

Due to Proposition \ref{po.find} and Lemma \ref{po.ppnd}, the assumptions
of the refined chain rule, Proposition \ref{rr.newcomp}, are satisfied 
with $\tilde{X} = C(I)$, $Y = \tilde{Y} = L^{\tilde{q}}(I)$ for some
$\infty > \tilde{q} > q$, $Z = L^q(I)$. This proves (\ref{po.ndip.6}).
The estimate (\ref{po.ndip.5}) follows from (\ref{po.ndip.1}) and 
(\ref{mam.l.1}), as
$\lambda_H$ takes values in $[0,1]$ and, 
setting $\lambda_t = \lambda_H(\tilde{F}^I(u,p)(t))$,
\[
L_+^I(h,\eta)(t) = (1-\lambda_t)\eta + \lambda_t \langle \mu^I(u,t) , h\rangle \,.
\]
Since $P_+^I$ is global Lipschitz continuous w.r.t. the maximum norm, 
the final assertion, too, follows in view of (\ref{po.ndip.5}).
\end{proof}
We also need a variant of the preceding proposition. For $I = [t_*,t^*]$ we define
\begin{equation}\label{po.ndipf.1}
P_{+,*}^I: Z_+^I\to\real \,,\quad P_{+,*}^I(u,p) = P_+^I(u,p)(t^*) \,.
\end{equation}
According to (\ref{po.ndip.1}), setting
\begin{equation}\label{po.ndipf.2}
L_{+,*}^I(h,y) = L_+^I(h,y)(t^*) \,,\quad L_+^I\in G_+^I(u,p) \,,
\end{equation}
yields a well-defined element $L_{+,*}^I\in (C(I)\times\real)^*$.
\begin{proposition}\label{po.ndipf}
The mapping
\begin{equation}\label{po.ndipf.4}
\begin{split}
G_{+,*}^I: Z_+^I\cap (X\times\real) \toto (X\times\real)^*
\qquad 
\\
G_{+,*}^I(u,p) = \{L_{+,*}^I: \text{$L_{+,*}^I$ given by (\ref{po.ndipf.2})} \}
\end{split}
\end{equation}
is a Newton derivative of $P_{+,*}^I: Z_+^I\cap (X\times\real) \to \real$.
The elements $L_{+,*}^I$ of $G_{+,*}^I(u,p)$ satisfy, 
for all $(u,p)\in Z_+^I\cap (X\times\real)$,
\begin{equation}\label{po.ndipf.5}
|L_+^I(h,\eta)| \le \max\{\|h\|_{\infty,t^*},|\eta|\}
\end{equation}
for all $h\in C(I)$, $\eta\in\real$. 
Moreover, for all such $(u,p)$ the remainder estimate
\begin{equation}\label{po.ndipf.6}
\begin{split}
&\sup_{L_{+,*}^I\in G_{+,*}^I(u+h,p+\eta)}
|P_{+,*}^I(u+h,p+\eta) - P_{+,*}^I(u,p) - L_{+,*}^I(h,\eta)| 
\\&\qquad \qquad  \le 
\rho_{(u,p)}(\|h\|_\infty + |\eta|) \|(h,\eta)\|_{X\times\real}
\end{split}
\end{equation}
holds for all $h\in X$ with $u+h\in Z_+^I$ and all $\eta\in\real$. 
The remainder term $\rho_{(u,p)}$ satisfies $\rho_{(u,p)}(\delta)\downarrow 0$
as $\delta\downarrow 0$ and is uniformly bounded in $(u,p)$.
\end{proposition}
\begin{proof}
We proceed in a manner analogous to the proof of Proposition \ref{po.ndip}. 
We apply Proposition \ref{rr.newcomp} to the decomposition
\[
P_{+,*}^I(u,p) = p + \max\{0,\max_I (u-p-r)\}
\]
The Newton derivative of the inner maximum satisfies the refined remainder
estimate given in Proposition \ref{nmf.nbder}. The outer maximum is just
the positive part mapping on $\real$.
\end{proof}
A Newton derivative $G_-^I$ of the mapping
\begin{equation}\label{po.ndin.0}
P_-^I(u,p) = p - F_{pp}(\tilde{F}^I(-u,-p))
\end{equation}
is obtained with analogous computations.
Its elements $L_-^I\in G_-^I(u,p)$ have the form
\begin{equation}\label{po.ndin.1}
L_-^I(h,\eta)(t) = \eta - \lambda_H(\tilde{F}^I(-u,-p)(t))
\cdot\langle \mu^I(-u,t) , -h+\eta \rangle
\end{equation}
with functions $\lambda_H\in S_H$ and measures $\mu^I\in S_\Phi^I$.
The associated set-valued mapping $\Psi_-^I$ and a set $S_-^I$ of
measurable selectors is given by
\begin{equation}\label{po.ndin.2}
\begin{split}
\Psi_-^I: C(I)\times\real\times I \toto (C(I)\times\real)^*
\qquad \qquad \qquad \qquad
\\
\Psi_-^I(u,p,t) = \pi_2 - \tilde{\Psi}^I_{pp}(-u,-p,t) \cdot (-\tilde{\Psi}^I(-u,-p,t))\,,
\qquad \qquad
\\
S_-^I = \{\tilde{\nu}: \tilde{\nu}(u,p,t) =
\pi_2 + \tilde{\lambda}(-u,-p,t)\tilde{\mu}(-u,-p,t),\,
\tilde{\mu}\in \tilde{S}_\Psi^I,\,\tilde{\lambda}\in\tilde{S}^I_{pp}\} \,.
\end{split}
\end{equation}
where $\pi_2$ again denotes the projection $\pi_2:C(I)\times\real\to\real$,
$\pi_2(h,\eta) = \eta$.

The analogue of Proposition \ref{po.ndipf} also holds on minus intervals.

We combine $G^I_\pm$ and $\Psi^I_\pm$ into mappings $G^I$ and $\Psi^I$.
Indeed, on $Z_+^I\cap Z_-^I$, we have
$\tilde{F}^I(u,p) < 0$ and $\tilde{F}^I(-u,-p) < 0$ by Lemma \ref{po.dec.wpr}.
Consequently,
\[
\tilde{\Psi}^I_{pp}(u,p,t) = \{0\} \,,\quad
\Psi_\pm^I(u,p,t) = \{\pi_2\}
\]
for all $(u,p)\in Z_+^I\cap Z_-^I$, $t\in I$.
The argument of $\lambda_H$
in the representations (\ref{po.ndip.1}) and (\ref{po.ndin.1}) is negative,
therefore $L_\pm^I(h,\eta)(t) = \eta$ on $I$. As the sets $Z_\pm^I$ are open
subsets of $Z^I$, from Proposition \ref{po.ndip} and
the corresponding result for $P_-^I$ we get the following result.
\begin{proposition}\label{po.ndi}
The mapping $\Psi^I: Z^I\times I \toto (C(I)\times\real)^*$ defined by
$\Psi^I = \Psi_\pm^I$ on $Z_\pm^I$ is well-defined and usc and has
$w^*$-compact values.
The set
\begin{equation}\label{po.ndi.1}
S^I = \{\nu: \text{$\nu(u,p,t) = \tilde{\nu}_\pm(u,p,t)$
with $\tilde{\nu}_\pm \in S_{\pm}^I$, $(u,p,t)\in Z_\pm^I\times I$}\}
\end{equation}
consists of measurable selectors of $\Psi^I$.
The mapping $G^I: Z^I\cap (X\times\real) \toto \mathcal{L}(X\times\real,L^q(I))$
given by $G^I = G_\pm^I$ on $Z_\pm^I\cap (X\times\real)$ is well-defined and is
a Newton derivative of $P^I: Z^I\cap (X\times\real)\to L^q(I)$.
The estimates (\ref{po.ndip.5}) and (\ref{po.ndip.6}) hold with
$G^I,P^I,Z^I$ and $L^I$ in place of $G_+^I,P_+^I,Z_+^I$ and $L_+^I$, respectively. 
The remainder term $\rho_{(u,p)}$ satisfies $\rho_{(u,p)}(\delta\downarrow 0$
as $\delta\downarrow 0$ and is uniformly bounded in $(u,p)$.
\hfill $\Box$
\end{proposition}


\textbf{The initial value.}
According to (\ref{po.0}), the initial value of the play is given by
\begin{equation}\label{po.iv.1}
w_0(u,z_0) = u(a) - \pi_r(z_0) = u(a) - \max\{-r,\min\{r,z_0\}\} \,.
\end{equation}
It is well known that the mapping $R:\real\toto\real$,
\begin{equation}\label{po.iv.2}
R(x) = \begin{cases} 0 \,,& |x| > r \,,\\ [0,1] \,,& |x| = r \,,\\ 1 \,, & |x| < r
\end{cases}
\end{equation}
is a Newton derivative of $\pi_r$ and that $R$ is usc.
Then
\begin{equation}\label{po.iv.3}
S_R = \{\lambda_0: \text{$\lambda_0$ selector of $R$, $\lambda_0(\pm r)\in\mathbb{Q}$}\}
\end{equation}
defines a countable family of measurable selectors of $R$.
We set
\begin{equation}\label{po.iv.5}
\begin{split}
\Psi_0: C[a,b]\times\real \toto (C[a,b]\times\real)^*
\\
\Psi_0(u,z_0) = \{\delta_a\}\times (-R(z_0)) \,.
\end{split}
\end{equation}
\begin{lemma}\label{po.wnull}
A Newton derivative of $w_0: C[a,b]\times\real\to\real$ is given by
\begin{equation}\label{po.wnull.1}
\begin{split}
G_0: C[a,b]\times\real \toto (C[a,b]\times\real)^* \,,
\\
G_0(u,z_0) = \{L: L(h,y) = h(a) - \lambda_0(z_0)y \,,\lambda_0\in S_R\} \,.
\end{split}
\end{equation}
We have
\begin{equation}\label{po.wnull.3}
|L(h,y)| \le \|h\|_\infty + |y|
\end{equation}
for all $L\in G_0(u,z_0)$ and all $(u,z_0)\in C[a,b]\times\real$.
\end{lemma}
\begin{proof}
Let $L\in G_0(u,z_0)$. Then for all $(h,y)\in C[a,b]\times\real$ we have
\begin{align}
&|w_0(u+h,z_0+y) - w_0(u,z_0) - L(h,y)| =
|(\pi_r(z_0+y) - \pi_r(z_0) - \lambda_0(z_0)y)|
\nonumber \\ &\qquad \qquad \label{po.wnull.4}
\le \rho(|y|)|y|
\end{align}
with some $\rho(\delta) \downarrow 0$ as $\delta\downarrow 0$, since $R$
is a Newton derivative of $\pi_r$.
\end{proof}

\textbf{A Newton derivative on a partition for small input oscillations.}

Let $\Delta = \{t_k\}_{0\le k\le N}$ be a partition of $[a,b]$.
According to Lemma \ref{po.decp}, on the set $Z^\Delta$ of small input
oscillations, see (\ref{po.part.1}), the play can be written as
a composition of the mappings $P^{I_k}$ which belong to the partition
intervals $I_k = [t_{k-1},t_k]$.
Consequently, we obtain a Newton derivative of the play on $Z^\Delta$
as a composition of the Newton derivatives of $P^{I_k}$ as follows.

We define $w_k^\Delta: Z^\Delta\to\real$ and $P_k^\Delta: Z^\Delta\to C(I_k)$,
setting $w_0^\Delta = w_0$ and $\Psi_0^\Delta = \Psi_0$ from (\ref{po.iv.1})
and (\ref{po.iv.5}), and for $k\ge 1$
\begin{equation}\label{po.wkd.0}
\begin{split}
w_k^\Delta(u,z_0) &= P^{I_k}(u,w_{k-1}^\Delta(u,z_0))(t_k) \,,
\\
P_k^\Delta(u,z_0)(t) &= P^{I_k}(u,w_{k-1}^\Delta(u,z_0))(t) \,,\quad t\in I_k \,.
\end{split}
\end{equation}
Using Lemma \ref{po.decp} successively we see that
$P_k^\Delta(u,z_0) = \playr[u;z_0]$ on $I_k$.

We define $\Psi_k^\Delta: Z^\Delta\times I_k \toto (C[a,b]\times\real)^*$ by
$\Psi_0^\Delta = \Psi_0$ and, for $k\ge 1$,
\begin{align}\label{po.psikd.1}
&\Psi_k^\Delta(u,z_0,t) =
\\ \nonumber &\quad \{L\circ (\pi_1,L_{k-1}):\,
L\in\Psi^{I_k}(u,w_{k-1}^\Delta(u,z_0),t),\,L_{k-1}\in\Psi_{k-1}^\Delta(u,z_0,t_{k-1})\}
\end{align}
Here $\pi_1$ denotes the projection $\pi_1:C[a,b]\times\real\to C(I_k)$,
$\pi_1(u,z_0) = u|I_k$.
The mappings $\Psi^{I_k}$ have been constructed in Proposition \ref{po.ndi}.
The elements $L_k$ of $\Psi_k^\Delta(u,z_0,t)$ have the form
\begin{equation}\label{po.psikd.2}
L_k(h,y) = L(h,L_{k-1}(h,y)) \,,\quad h\in C[a,b],\,y\in\real \,.
\end{equation}
The sets $S_0^\Delta = S_R$,
\begin{equation}\label{po.psikd.3}
\begin{split}
S_k^\Delta &= \{\mu_k^\Delta:
\mu_k^\Delta(u,z_0,t) = \nu(u,p,t)\circ (\pi_1,\mu_{k-1}^\Delta(u,z_0,t_{k-1})),
\\ & \qquad \qquad
\nu\in S^{I_k},\,p = w_{k-1}^\Delta(u,z_0),\,
\mu_{k-1}^\Delta\in S_{k-1}^\Delta \} \,,\quad k\ge 1 \,,
\end{split}
\end{equation}
consist of measurable selectors of $\Psi_k^\Delta$.

We define $W_0^{\Delta} = G_0$ and inductively for $k\ge 1$
\begin{equation}\label{po.wkd.1}
\begin{split}
& \qquad \qquad \qquad
W_k^\Delta: Z^\Delta \toto (C[a,b]\times\real)^* \,,
\\
&W_k^\Delta(u,z_0) = \{ L_k^w: L_k^w = \mu_k^\Delta(u,z_0,t_k)
\text{ with $\mu_k^\Delta\in S_k^\Delta$}\} \,.
\end{split}
\end{equation}
The elements $L_k^w\in W_k^\Delta(u,z_0)$ satisfy
\begin{equation}\label{po.wkd.2}
\begin{split}
& \qquad \qquad
L_k^w(h,y) = L^{I_k}(h,L_{k-1}^w(h,y))(t_k) \,,
\\
&L^{I_k}\in G^{I_k}(u,w_{k-1}(u,z_0))\,,\quad
L_{k-1}^w\in W_{k-1}^\Delta(u,z_0) \,.
\end{split}
\end{equation}
We define $G_0^\Delta = G_0$ and inductively for $k\ge 1$
\begin{gather}\label{po.gkd.1}
\begin{split}
G_k^\Delta: Z^\Delta \toto \mathcal{L}(C[a,b]\times\real,L^\infty(I_k)) \,,
\qquad \qquad
\\
G_k^\Delta(u,z_0) = \{ L_k^\Delta: L_k^\Delta \text{ satisfies (\ref{po.gkd.2})
for some $\mu_k^\Delta \in S_k^\Delta$}\}
\,,
\end{split}
\\ \label{po.gkd.2}
L_k^\Delta(h,y)(t) = \langle \mu_k^\Delta(u,z_0,t)\,,\,(h,y)\rangle \,.
\end{gather}
The mappings $L_k^\Delta$ satisfy
\begin{equation}\label{po.gkd.3}
\begin{split}
& \qquad \quad
L_k^\Delta(h,y)(t) = L^{I_k}(h,L_{k-1}^w(h,y))(t) \,,
\\
&L^{I_k}\in G^{I_k}(u,w_{k-1}(u,z_0))\,,\quad
L_{k-1}^w\in W_{k-1}^\Delta(u,z_0) \,.
\end{split}
\end{equation}

\begin{proposition}\label{po.gwknd}
Let $0\le k\le N$, $1\le q < \infty$. \hfill\\
(i) The mapping
$\Psi_k^\Delta: Z^\Delta\times I_k \toto (C[a,b]\times\real)^*$ is usc
and has $w^*$-compact values, $S_k^\Delta$ is a set of measurable selectors
of $\Psi_k^\Delta$.
\hfill\\
(ii) The mapping
$W_k^\Delta: Z^\Delta\cap(X\times\real) \toto (C[a,b]\times\real)^*$
is a Newton derivative of $w_k^\Delta: Z^\Delta\cap (X\times\real)\to\real$.
The elements $L_k^w$ of $W_k^\Delta$ satisfy the estimate
\begin{equation}\label{po.gwknd.01}
|L_k^w(h,y)| \le \|h\|_{\infty,t_k} + |y|
\end{equation}
for all $h\in C[a,b]$ and $y\in\real$, uniformly in $(u,z_0)$.
Moreover, for all such $(u,z_0)$ the remainder estimate
\begin{equation}\label{po.gwknd.02}
\begin{split}
&\sup_{L_k^w\in W_k^\Delta(u+h,z_0+y)}
|w_k^\Delta(u+h,z_0+y) - w_k^\Delta(u,z_0) - L_k^w(h,y)|
\\&\qquad \qquad  \le 
\rho_{(u,z_0)}(\|h\|_\infty + |y|) \|(h,y)\|_{X\times\real}
\end{split}
\end{equation}
holds for all $h\in X$ with $u+h\in Z_+^I$ and all $y\in\real$. 
The remainder term $\rho_{(u,z_0)}$ satisfies $\rho_{(u,z_0)}(\delta)\downarrow 0$
as $\delta\downarrow 0$ and is uniformly bounded in $(u,z_0)$.
\hfill\\ 
(iii) The mapping
$G_k^\Delta: Z^\Delta\cap(X\times\real) \toto \mathcal{L}(X\times\real,L^q(I_k))$
is a Newton derivative of $P_k^\Delta$.
The elements $L_k^\Delta$ of $G_k^\Delta(u,z_0)$ satisfy the estimate
\begin{equation}\label{po.gwknd.1}
\|L_k^\Delta(h,y)\|_{\infty,t} \le \|h\|_{\infty,t} + |y|
\end{equation}
for all $h\in C[a,b]$ and $y\in\real$, uniformly in $(u,z_0)$.
Moreover, for all such $(u,z_0)$ the remainder estimate
\begin{equation}\label{po.gwknd.2}
\begin{split}
&\sup_{L_k^w\in W_k^\Delta(u+h,z_0+y)}
|P_k^\Delta(u+h,z_0+y) - P_k^\Delta(u,z_0) - L_k^\Delta(h,y)|
\\&\qquad \qquad  \le 
\rho_{(u,z_0)}(\|h\|_\infty + |y|) \|(h,y)\|_{X\times\real}
\end{split}
\end{equation}
holds for all $h\in X$ with $u+h\in Z_+^I$ and all $y\in\real$. 
The remainder term $\rho_{(u,z_0)}$ satisfies $\rho_{(u,z_0)}(\delta)\downarrow 0$
as $\delta\downarrow 0$ and is uniformly bounded in $(u,z_0)$.
\hfill\\ 
\end{proposition}

\begin{proof}
We proceed by induction over $k$. The case $k=0$ is treated in Lemma \ref{po.wnull}.
Now assume the result is proved for $k-1$.
\hfill\\
(i) We apply Proposition \ref{svm.psicomp}, setting there $J = [a,b]$,
$I = I_k$, $U = C^\Delta$ and
\begin{gather*}
F_1(u,z_0,t) = w_{k-1}^\Delta(u,z_0) \,,\quad
F_1:C^\Delta\times\real\times I_k \to \real \,,
\\
F_2(u,z_0,p,t) = P^{I_k}(u,p)(t) \,,\quad
F_2:C^\Delta\times\real\times\real\times I_k\to \real \,,
\\
\Psi_1(u,z_0,t) = \Psi_{k-1}^\Delta(u,z_0,t_{k-1}) \,,\quad
\Psi_2(u,z_0,p,t) = \Psi^{I_k}(u,p,t) \,,\quad
\Psi = \Psi_k^\Delta \,.
\end{gather*}
(ii) We apply Proposition \ref{rr.newcomp} to the decomposition
\[
(u,z_0) \mapsto (u,w_{k-1}^\Delta(u,z_0)) \mapsto w_k^\Delta(u,z_0)
\]
given by the first equation in (\ref{po.wkd.0}).  
Its assumptions are satisfied by the induction hypothesis and by
Proposition \ref{po.ndipf}. (\ref{po.gwknd.01}) follows from the
estimate
\begin{align*}
|L_k^w(h,y)| &\le \|L^{I_k}(h,L_{k-1}^w(h,y))\|_{\infty,t_k}
\le \max\{\|h\|_{\infty,t_k},\|h\|_{\infty,t_k} + |y|\} 
\\ &= 
\|h\|_{\infty,t_k} + |y| \,.
\end{align*}
(iii) This follows as in (ii), using Proposition \ref{po.ndi} instead of
Proposition \ref{po.ndipf}, as well as the estimate 
\begin{align*}
\|L_k^\Delta(h,y)\|_{\infty,t} &= \|L^{I_k}(h,L_{k-1}^w(h,y))\|_{\infty,t}
\le \max\{\|h\|_{\infty,t},\|h\|_{\infty,t} + |y|\} 
\\ &= 
\|h\|_{\infty,t} + |y| \,.
\end{align*}
\end{proof}

\textbf{A Newton derivative of the play on the whole space {\boldmath $X\times\real$.}}

Let $\{\Delta_n\}$ be a sequence of partitions of $[a,b]$
such that $|\Delta_n|\to 0$ as $n\to\infty$ and that $\Delta_{n+1}$ is obtained
from $\Delta_n$ by adding a single point $t\notin\Delta_n$, starting from
$\Delta_1 = \{a,b\}$. We have
\begin{equation}\label{po.ws.0}
Z^{\Delta_{n}} \subset Z^{\Delta_{n+1}} \,,\quad
C[a,b]\times\real = \bigcup_{n\in\nat} Z^{\Delta_n}
\end{equation}
and consequently
\begin{equation}\label{po.ws.1}
X\times \real = \bigcup_{n\in\nat} (Z^{\Delta_n}\cap (X\times\real))
= \bigcup_{n\in\nat} (X^{\Delta_n}\times\real) \,.
\end{equation}
We construct a Newton derivative $G$ of the play on $X\times\real$
from the Newton derivatives $G_k^{\Delta_n}$ of $P_k^{\Delta_n}$ obtained
in Proposition \ref{po.gwknd}. This is done in two steps. In the first step, we
define a Newton derivative $G^{\Delta_n}$ of the play on
$Z^{\Delta_n}\cap (X\times\real)$; in the second step we glue together
these derivatives using Proposition \ref{nwd.part}.

Since the elements of $G_k^{\Delta_n}$ are obtained as measurable selectors
of the set-valued mappings $\Psi_k^{\Delta_n}$, we first
combine those to a mapping $\Psi^{\Delta_n}$.
In order to do this, we consider the following situation.

Let $\Delta = \{t_j\}$, $a = t_0 < \dots < t_N = b$, be a partition of $[a,b]$,
let $I_j = [t_{j-1},t_j]$ for $1\le j\le N$, where $N\ge 1$.
If $N > 1$, let $\Delta'$ be the partition which results from $\Delta$ when
we remove a single point $t_k$, $k\in \{1,\dots,N-1\}$.
\begin{proposition}\label{po.psid}
(i) If $N > 1$,
\begin{equation}\label{po.psid.2}
\Psi_j^{\Delta}(u,z_0,t) =  \Psi^{\Delta'}(u,z_0,t)
\end{equation}
for all $(u,z_0)\in Z^{\Delta'},\, t\in I_j,\,1\le j \le N$. \hfill\\
(ii) The mapping
$\Psi^\Delta: Z^\Delta\times [a,b] \toto (C[a,b]\times\real)^*$ defined by
\begin{equation}\label{po.psid.1}
\Psi^\Delta(u,z_0,t) = \Psi_j^\Delta(u,z_0,t) \,,\quad \text{if $t \in I_j$,}
\end{equation}
is well-defined, usc and has $w^*$-compact values. \hfill\\
(iii) The set
\begin{equation}\label{po.psid.3}
\begin{split}
S^\Delta &= \{\mu^\Delta: \text{$\mu^\Delta$ selector of $\Psi^\Delta$},\,
\mu^\Delta = \mu_1^\Delta \text{ on } Z^\Delta\times [t_0,t_1],\,
\\ &\qquad \qquad
\mu^\Delta = \mu_k^\Delta \text{ on } Z^\Delta\times (t_{k-1},t_k]
\text{ for $k>1$},\, \mu_k^\Delta \in S_k^\Delta \}
\end{split}
\end{equation}
consists of measurable selectors of $\Psi^\Delta$.
\end{proposition}
\begin{proof}
The proof proceeds by induction on $N$, the number of partition points
of $\Delta$ being equal to $N+1$.
For $N=1$ we have $\Psi^\Delta = \Psi_1^\Delta$, (i) is empty, and (ii)
has been proved already in Proposition \ref{po.gwknd}.
For the induction step $N-1 \to N$ we assume that $\Psi^{\Delta'}$ is
well-defined.

In order to prove (\ref{po.psid.2}), it suffices to show that
\begin{equation}\label{po.psid.6}
\Psi_{k+1}^\Delta = \Psi_k^{\Delta'} \quad \text{on $Z^{\Delta'}\times I_{k+1}$.}
\end{equation}
Indeed, as the refinement by adjoining $\{t_k\}$ to $\Delta'$ does not change the
partition intervals $I_j$ contained in $[a,t_{k-1}]$ and in $[t_{k+1},b]$, we have
$\Psi_j^\Delta = \Psi_j^{\Delta'}$ on $Z^{\Delta'}\times I_j$ for
all $j\le k$, and (once we have shown that
$\Psi_{k+1}^\Delta = \Psi_k^{\Delta'}$ on $Z^{\Delta'}\times \{t_{k+1}\}$) also
$\Psi_{j+1}^\Delta = \Psi_j^{\Delta'}$ on $Z^{\Delta'}\times I_{j+1}$
for all $k < j < N$. (\ref{po.psid.6}) will be proved in Lemma \ref{po.teil}
below.

We now prove (ii). For $(u,z_0)\in Z^{\Delta'}$, (\ref{po.psid.1}) holds
by (i). Let $(u,z_0)\in Z^\Delta$. We define $\tilde{u}\in C[a,b]$ by
\[
\tilde{u}(t) = \begin{cases} u(t) \,,& t\in [a,t_{N-1}] \,,\\
u(t_{N-1}) \,, & t\in I_N = [t_{N-1},b] \,.
\end{cases}
\]
Then $(\tilde{u},z_0)\in Z^{\Delta'}$. By (i),
we have
\[
\Psi_j^{\Delta}(u,z_0,t) = \Psi_j^{\Delta}(\tilde{u},z_0,t)
= \Psi^{\Delta'}(\tilde{u},z_0,t)
\]
holds for all $t\in I_j]$, $1\le j \le N-1$.
In particular,
\[
\Psi_{j+1}^\Delta(u,z_0,t_j) = \Psi_j^\Delta(u,z_0,t_j)
\quad \text{for all $1\le j \le N-1$.}
\]
This shows that $\Psi^\Delta$ is well-defined by (\ref{po.psid.1}) if
$(u,z_0)\in Z^\Delta$.
That $\Psi^\Delta$ is usc and has $w^*$-compact values now follows from
Lemma \ref{svm.uscext}.

To prove (iii) it suffices to observe that the functions $\mu_k^\Delta$
are measurable.
\end{proof}

The proof of Proposition \ref{po.psid} is based on Lemma \ref{po.teil} which
in turn is based on Lemma \ref{po.cases}. The proof of Lemma \ref{po.cases}
only uses results derived before and up to Proposition \ref{po.gwknd}.

\begin{lemma}\label{po.teil}
We have
\begin{equation}\label{po.teil.1}
\Psi_{k+1}^\Delta(u,z_0,t) = \Psi_k^{\Delta'}(u,z_0,t)
\end{equation}
for all $(u,z_0)\in Z^{\Delta'}$ and all $t\in I_{k+1}$.
\end{lemma}

\begin{proof}
Let $(u,z_0)\in Z^{\Delta'}$ and $t\in I_{k+1} = [t_k,t_{k+1}]$ be given,
set $I' = [t_{k-1},t]$. Moreover, set $w = \mathcal{P}_r[u;z_0]$,
$p_{k-1} = w(t_{k-1})$ and $p_k = w(t_k)$.
We assume that $(u,p_{k-1})\in Z_+^{I'}$, that is, $I'$ is a plus interval
by (\ref{po.decp.20}).
(The case of a minus interval is treated analogously.)
By (\ref{po.psikd.1}), the elements of $\Psi_k^{\Delta'}(u,z_0,t)$ have the form
$L'\circ (\pi_1,L_{k-1})$ with
\begin{equation}\label{po.teil.5}
L'\in \Psi_+^{I'}(u,p_{k-1},t) =: \Psi' \,,\quad
L_{k-1}\in \Psi_{k-1}^{\Delta'}(u,z_0,t_{k-1}) \,.
\end{equation}
For the same reason, the elements of $\Psi_{k+1}^\Delta(u,z_0,t)$
have the form $L\circ (\pi_1,L_k)$ with
\begin{equation}\label{po.teil.6}
L\in \Psi_+^{I_{k+1}}(u,p_k,t) =: \Psi \,,\quad
L_k\in \Psi_k^\Delta(u,z_0,t_k) \,,
\end{equation}
and the elements $L_k$ of $\Psi_k^\Delta(u,z_0,t_k)$ have the form
$L''\circ (\pi_1,L_{k-1})$ with
\begin{equation}\label{po.teil.7}
L''\in \Psi_+^{I_k}(u,p_k,t_k) =: \Psi'' \,,\quad
L_{k-1}\in \Psi_{k-1}^{\Delta}(u,z_0,t_{k-1}) \,.
\end{equation}
Since $\Psi_{k-1}^\Delta = \Psi_{k-1}^{\Delta'}$ on $Z^{\Delta'}\times I_{k-1}$,
in order to prove (\ref{po.teil.1}) it suffices to prove that
\begin{equation}\label{po.teil.10}
\Psi' = \{L\circ (\pi_1,L''): L\in\Psi,\,L''\in\Psi''\} \,.
\end{equation}
This is done in Lemma \ref{po.cases} below.
\end{proof}

Let $A_1,A_2,A_3$ be sets, let $\mathcal{F}_i$ be sets of mappings from $A_i$
to $A_{i+1}$, $i=1,2$. We define the elementwise composition
\begin{equation}\label{po.cases.0}
\mathcal{F}_2 \circ \mathcal{F}_1 = \{ f_2\circ f_1: f_1\in \mathcal{F}_1,\,
f_2\in \mathcal{F}_2\} \,.
\end{equation}
We also provide a more explicit representation of $\Psi_+^I$.
Inserting (\ref{po.pi.6}) and (\ref{po.pp.5}) into (\ref{po.ndip.2}) yields
\begin{equation}\label{po.ndip.8}
\Psi_+^I(u,p,t) = \pi_2 + H(\tilde{F}^I(u,p)(t))\cdot(\Phi^I(u,t)\circ\pi_1 - \pi_2)\,.
\end{equation}
Here, $\Phi^I(u,t) \circ \pi_1 := \{L\circ \pi_1: L\in\Phi^I(u,t)\}$.

Setting $\sigma = \tilde{F}^I(u,p)(t)$ we get
\begin{equation}\label{po.ndip.9}
\Psi_+^I(u,p,t) = \begin{cases}
\{\pi_2\}\,, & \sigma < 0 \,, \\
\Phi^I(u,t)\circ\pi_1\,, & \sigma > 0 \,, \\
{\rm co}(\{\pi_2\}\cup (\Phi^I(u,t)\circ\pi_1))\,, & \sigma = 0 \,.
\end{cases}
\end{equation}

\begin{lemma}\label{po.cases}
Let $\Psi,\Psi'$ and $\Psi''$ be the subsets of $(C[a,b]\times\real)^*$ defined
in (\ref{po.teil.5}) -- (\ref{po.teil.7}). Then
\begin{equation}\label{po.cases.1}
\Psi' = \Psi\circ (\pi_1,\Psi'') \,.
\end{equation}
\end{lemma}

\begin{proof}
We continue to use the notations from the proof of Lemma \ref{po.teil}; again,
the case of a minus interval is treated analogously. We define
\begin{equation}\label{po.cases.2}
\sigma' = F^{I'}(u-p_{k-1}-r)(t)
\end{equation}
and obtain from (\ref{po.ndip.9})
\begin{equation}\label{po.cases.3}
\Psi' = \begin{cases}
\{\pi_2\}\,, & \sigma' < 0 \,, \\
\Phi^{I'}(u,t)\circ\pi_1\,, & \sigma' > 0 \,, \\
{\rm co}(\{\pi_2\}\cup (\Phi^{I'}(u,t)\circ\pi_1))\,, & \sigma = 0 \,.
\end{cases}
\end{equation}
For $\Psi$ and $\Psi''$ corresponding formulas hold; we replace $\sigma'$
with
\begin{equation}\label{po.cases.4}
\sigma = F^{I_{k+1}}(u-p_k-r)(t) \,,\quad
\sigma'' = F^{I_k}(u-p_{k-1}-r)(t_k) \,,
\end{equation}
and $\Phi^{I'}(u,t)$ by $\Phi^{I_{k+1}}(u,t)$ resp. $\Phi^{I_k}(u,t_k)$.
From the definitions we immediately see that $\sigma'' \le \sigma'$ and
\begin{equation}\label{po.cases.6}
p_k = p_{k-1} \quad \Rightarrow \quad \sigma' = \max\{\sigma'',\sigma\} \,,
\end{equation}
as well as
\begin{align}
\label{po.cases.71}
\sigma \le 0 \quad &\Rightarrow \quad w = p_k \quad \text{on $[t_k,t]$,}
\\
\label{po.cases.72}
\sigma' \le 0 \quad &\Rightarrow \quad w = p_{k-1} = p_k \quad \text{on $I'$,}
\\
\label{po.cases.73}
\sigma'' \le 0 \quad &\Rightarrow \quad w = p_{k-1} \quad \text{on $I_k$.}
\end{align}
In order to prove (\ref{po.cases.1}), we have to distinguish several cases.
\hfill\\ \textit{Case 1: $\sigma' < 0$.}
Then $\sigma < 0$ and $\sigma'' < 0$ by (\ref{po.cases.72}) and (\ref{po.cases.6}),
so $\Psi = \Psi' = \Psi'' = \{\pi_2\}$, and (\ref{po.cases.1}) holds.
\hfill\\ \textit{Case 2: $\sigma' = 0$.}
As in Case 1, we have
\begin{equation}\label{po.cases.21}
w = p_{k-1} = p_k \quad \text{on $I'$,} \qquad
0 = \sigma' = \max\{\sigma'',\sigma\} \,.
\end{equation}
\textit{Subcase 2a: $\sigma'' < 0$.}
Then $\sigma'' < 0 = \sigma$, so $(F^{I_k}u)(t_k) < (F^{I_{k+1}}u)(t)$ since
$p_{k-1} = p_k$. Therefore,
\[
M^{I'}(u,t) = M^{I_{k+1}}(u,t) \,,\quad
\Phi^{I'}(u,t) = \Phi^{I_{k+1}}(u,t) \,,\quad \Psi' = \Psi \,.
\]
As $\Psi'' = \{\pi_2\}$, (\ref{po.cases.1}) holds.
\hfill\\ \textit{Subcase 2b: $\sigma'' = 0$.}
By Proposition \ref{po.decbp}(iii), there exists $\tau''\in I_k$ such that
\begin{equation}\label{po.cases.23}
u(\tau'') - w(\tau'') = r = u(\tau'') - p_k \,.
\end{equation}
\textit{Subsubcase 2b1: $\sigma < 0$.}
Then on $[t_k,t]$ we have
$u \le F^{I_{k+1}}(u) < r + p_k = u(\tau'')$,
so
\[
M^{I'}(u,t) = M^{I_k}(u,t_k) \,,\quad
\Phi^{I'}(u,t) = \Phi^{I_k}(u,t_k) \,,\quad \Psi' = \Psi'' \,.
\]
As $\Psi = \{\pi_2\}$, (\ref{po.cases.1}) holds.
\hfill\\ \textit{Subsubcase 2b2: $\sigma = 0$.}
By Proposition \ref{po.decbp}(iii), there exists $\tau\in [t_k,t]$ such that
$u(\tau) - w(\tau) = r$.
Since $w$ is constant on $I'$, we have
\[
\tau\in M^{I'}(u,t) \,,\quad M^{I'}(u,t) \cap I_{k+1} = M^{I_{k+1}}(u,t) \,.
\]
For the same reason,
\[
\tau''\in M^{I'}(u,t) \,,\quad M^{I'}(u,t) \cap I_k = M^{I_k}(u,t_k) \,.
\]
This gives
\begin{equation}\label{po.cases.26}
M^{I'}(u,t) =  M^{I_k}(u,t_k) \cup M^{I_{k+1}}(u,t) \,.
\end{equation}
Setting
\[
\hat{\Phi} = \Phi^{I_{k+1}}(u,t) \,,\quad \hat{\Phi}' = \Phi^{I'}(u,t) \,,\quad
\hat{\Phi}'' = \Phi^{I_k}(u,t_k) \,,
\]
it follows from (\ref{po.cases.26}) that
\[
\hat{\Phi}'(u,t) = {\rm co}\,(\hat{\Phi}'' \cup \hat{\Phi}) \,.
\]
Since the sets involved are convex and the mappings involved are linear, we
can compute
\begin{align*}
\Psi \circ (\pi_1,\Psi'') &=
{\rm co}\,(\{\pi_2\}\cup (\hat{\Phi}\circ \pi_1)) \circ (\pi_1,\Psi'')
= {\rm co}\,(\Psi'' \cup (\hat{\Phi}\circ \pi_1))
\\ &=
{\rm co}\,(\{\pi_2\}\cup (\hat{\Phi}''\circ \pi_1)
\cup (\hat{\Phi}\circ \pi_1))
= {\rm co}\,(\{\pi_2\}\cup (\hat{\Phi}'\circ \pi_1))
= \Psi' \,.
\end{align*}
\textit{Case 3: $\sigma' > 0$.}
We have
\begin{equation}\label{po.cases.31}
p_{k-1} + \sigma' = w(t) = p_k + \max\{0,\sigma\} \,.
\end{equation}
Therefore $w(t) > p_{k-1}$,
and by Proposition \ref{po.decbp}(iii), there exists $\tau'\in I'$ such that
$u(\tau') - w(\tau') = r$.
\hfill\\ \textit{Subcase 3a: $\sigma > 0$.}
Then $w(t) > p_k = w(t_k)$, so $M^{I'}(u,t) = M^{I_{k+1}}(u,t)$ and therefore
\[
\Psi = \Phi^{I_{k+1}}(u,t)\circ\pi_1 = \Phi^{I'}(u,t)\circ\pi_1 = \Psi' \,.
\]
Since moreover $\Psi\circ(\pi_1,\Psi'') = \Psi$, (\ref{po.cases.1}) holds.
\hfill\\ \textit{Subcase 3b: $\sigma \le 0$.}
Then $p_{k-1} < w(t) = p_k = p_{k-1} + \max\{0,\sigma''\}$,
so $\sigma'' > 0$ and $\Psi'' = \Phi^{I_k}(u,t_k)\circ\pi_1$.
\hfill\\ \textit{Subsubcase 3b1: $\sigma < 0$.}
On $[t_k,t]$ we have $u-w \le (F^{I_{k+1}}u) - w < r$.
Since $u(\tau') - w(\tau') = r$, it follows that $M^{I'}(u,t) = M^{I_k}(u,t_k)$.
As $\Psi = \{\pi_2\}$,
\[
\Psi\circ(\pi_1,\Psi'') =
\Psi'' = \Phi^{I_k}(u,t_k)\circ\pi_1 = \Phi^{I'}(u,t)\circ\pi_1 = \Psi' \,.
\]
\textit{Subsubcase 3b2: $\sigma = 0$.}
By Proposition \ref{po.decbp}(iii), there exists $\tau\in [t_k,t]$ such that
$r = u(\tau) - w(\tau) = u(\tau) - p_k$. On $I'$ we have
\[
u \le w + r \le w(\tau) + r = u(\tau) \,,
\]
since $w$ is nondecreasing on $I'$ and constant on $[t_k,t]$. This implies, since
$p_{k-1} < p_k$ and $\sigma = 0$,
\begin{align*}
p_k &= w(t_k) = \max\{p_{k-1},F^{I_k}(u-r)(t_k)\} = F^{I_k}(u-r)(t_k) = u(\tau) - r
\\ &\le F^{I_{k+1}}(u-r)(t) = p_k \,,
\end{align*}
so
\[
(F^{I_{k+1}}u)(t) = (F^{I_k}u)(t_k) = (F^{I'}u)(t) \,.
\]
Consequently,
\[
M^{I'}(u,t) =  M^{I_k}(u,t_k) \cup M^{I_{k+1}}(u,t) \,.
\]
The proof of (\ref{po.cases.1}) now proceeds as in Subsubcase 2b2, the final
computation being modified to
\begin{align*}
\Psi \circ (\pi_1,\Psi'') &=
{\rm co}\,(\{\pi_2\}\cup (\hat{\Phi}\circ \pi_1)) \circ (\pi_1,\Psi'')
= {\rm co}\,(\Psi'' \cup (\hat{\Phi}\circ \pi_1))
\\ &=
{\rm co}\,((\hat{\Phi}''\circ \pi_1) \cup (\hat{\Phi}\circ \pi_1))
= \hat{\Phi}'\circ \pi_1 = \Psi' \,.
\end{align*}
\end{proof}

In order to define a Newton derivative $G^\Delta$ of the play on
$Z^\Delta \cap (X\times\real)$, we set
\begin{equation}\label{po.gdnd.0}
\begin{split}
G^\Delta: Z^\Delta \toto \mathcal{L}(C[a,b]\times\real,L^\infty(a,b)) \,,
\qquad \qquad
\\
G^\Delta(u,z_0) = \{L^\Delta:
L^\Delta(h,y)(t) = \langle \mu^\Delta(u,z_0,t),(h,y)\rangle
\\
\text{ on $[a,b]$ for some $\mu^\Delta \in S^\Delta$} \} \,.
\qquad
\end{split}
\end{equation}
\begin{proposition}\label{po.gdnd}
Let $1\le q < \infty$. \hfill\\
The mapping
$G^\Delta: Z^\Delta\cap(X\times\real) \toto \mathcal{L}(X\times\real,L^q(a,b))$
is a Newton derivative of the play
$\mathcal{P}_r: Z^\Delta\cap(X\times\real) \to L^q(a,b)$.
The elements $L^\Delta$ of $G^\Delta(u,z_0)$ satisfy the estimate
\begin{equation}\label{po.gdnd.1}
\|L^\Delta(h,y)\|_{\infty,t} \le \|h\|_{\infty,t} + |y|
\end{equation}
for all $h\in C[a,b]$ and $y\in\real$, uniformly in $(u,z_0)$.
Moreover, for all such $(u,z_0)$ the remainder estimate
\begin{equation}\label{po.gdnd.2}
\begin{split}
&\sup_{L^\Delta\in G^\Delta(u+h,z_0+y)}
\|\mathcal{P}_r[u+h;z_0+y] - \mathcal{P}_r[u;z_0] - L^\Delta(h,y)\|_{L^q(a,b)}
\\&\qquad \qquad  \le 
\rho_{(u,z_0)}(\|h\|_\infty + |y|) \|(h,y)\|_{X\times\real}
\end{split}
\end{equation}
holds for all $h\in X$ with $u+h\in Z_+^I$ and all $y\in\real$,
where $\rho_{(u,z_0)}(\delta)\downarrow 0$ as $\delta\downarrow 0$
and $\rho_{(u,z_0)}$ is uniformly bounded in $(u,z_0)$.
\end{proposition}
\begin{proof}
Let $(u,z_0)\in Z^\Delta \cap (X\times\real)$, $(h,y)\in X\times\real$
with $\|h\|_X$ small enough, and $L^\Delta\in G^\Delta(u+h,z_0+y)$.
Since $\mathcal{P}_r[u;z_0] = P_k^\Delta(u,z_0)$ and
$\mathcal{P}_r[u+h;z_0+y] = P_k^\Delta(u+h,z_0+y)$ on $I_k$, we have in view
of the definition of $S^\Delta$, $S_k^\Delta$ and $G_k^\Delta$
\begin{align*}
&\|\mathcal{P}_r[u+h;z_0+y] - \mathcal{P}_r[u;z_0] - L^\Delta(h,y)\|^q_{L^q(a,b)}
\\ &\qquad = \sum_{k=1}^N
\|P_k^\Delta(u+h,z_0+y) - P_k^\Delta(u,z_0) - L_k^\Delta(h,y)\|^q_{L^q(I_k)}
\end{align*}
for some $L_k^\Delta\in G_k^\Delta(u+h,z_0+y)$. As $G_k^\Delta$ is a Newton
derivative of $P_k^\Delta$ by Proposition \ref{po.gwknd}, (\ref{po.gwknd.1})
holds for $L_k^\Delta$, and (\ref{po.gwknd.2}) holds for the remainder,
the claim follows.
\end{proof}
On the basis of Proposition \ref{nwd.part}, we now construct a Newton derivative
of the play on $X\times\real$. We set
\[
U = X\times\real \,,\quad U_n = X^{\Delta_n}\times\real \,.
\]
Let $0 < r_1 < r_2 < \dots$ be an increasing sequence of positive numbers with
$r_n < r$ for all $n$. Let $\{I_{n,k}\}$ be the partition intervals of $\Delta_n$.
We define
\begin{equation}\label{po.gnd.0}
V_n = \{(u,z_0): u\in X,\,z_0\in\real,\, \osc_{I_{n,k}} u < r_n \text{ for all $k$}\} \,.
\end{equation}
Since $r_n < r_{n+1} < r$, we have $\overline{V}_n \subset U_n \cap V_{n+1}$.
Moreover, by (\ref{po.ws.1})
\[
\bigcup_n V_n = X\times\real = U\,,\quad 
\text{because $|\Delta_n| \to 0$ as $n\to\infty$.}
\]
Thus, all assumptions of Proposition \ref{nwd.part} are satisfied.

We finally arrive at the main result.
\begin{theorem}\label{po.gnd}
Let $1\le q < \infty$. \hfill\\
The mapping $G^{P_r}: X\times\real \toto \mathcal{L}(X\times\real,L^q(a,b))$
defined by
\begin{equation}\label{po.gnd.1}
G^{P_r}(u,z_0) = G^{\Delta_n}(u,z_0) \quad
\text{if } (u,z_0) \in \overline{V}_n \setminus \overline{V}_{n-1} \,,
\end{equation}
is a Newton derivative of the play
$\mathcal{P}_r: X\times\real \to L^q(a,b)$
with the remainder estimate
\begin{equation}\label{po.gnd.2}
\begin{split}
&\sup_{L^{\mathcal{P}_r}\in G^\Delta(u+h,z_0+y)}
\|\mathcal{P}_r[u+h;z_0+y] - \mathcal{P}_r[u;z_0] - L^{\mathcal{P}_r}(h,y)\|_{L^q(a,b)}
\\&\qquad \qquad  \le 
\rho_{(u,z_0)}(\|h\|_\infty + |y|) \|(h,y)\|_{X\times\real}
\end{split}
\end{equation}
where $\rho_{(u,z_0)}(\delta)\downarrow 0$ as $\delta\downarrow 0$. 
The elements $L^{P_r}$ of $G^r(u,z_0)$ satisfy the estimate
\begin{equation}\label{po.gnd.3}
\|L^{P_r}(h,y)\|_{\infty,t} \le \|h\|_{\infty,t} + |y|
\end{equation}
for all $h\in X$ and $y\in\real$, uniformly in $(u,z_0)$.
They have the form
\begin{equation}\label{po.gnd.4}
L^{P_r}(h,y)(t) = \langle \mu^{P_r}(u,z_0,t), (h,y)\rangle \,,\quad t\in (a,b)\,,
\end{equation}
with
\begin{equation}\label{po.gnd.5}
\mu^{P_r}(u,z_0,t) = \mu^{\Delta_n}(u,z_0,t) \quad
\text{if } (u,z_0) \in \overline{V}_n \setminus \overline{V}_{n-1} \,,\quad
\mu^{\Delta_n} \in S^{\Delta_n} \,.
\end{equation}
The functions $\mu^{P_r}:C[a,b]\times\real\times [a,b]\to (C[a,b]\times\real)^*$
are measurable.
\hfill $\Box$
\end{theorem}
\begin{proof}
This follows from Proposition \ref{nwd.part} and Proposition \ref{po.gdnd}
when we choose
\[
\rho_{u,z_0} = \max\{\rho_{n,u,z_0},\rho_{n+1,u,z_0}\} \quad
\text{if } (u,z_0) \in \overline{V}_n \setminus \overline{V}_{n-1} 
\]
with the remainder terms $\rho_{n,u,z_0}$ belonging to $G^{\Delta_n}$.
\end{proof}
Since the stop operator is related to the play operator by the formula
$\stopr[u;z_0] = u - \playr[u;z_0]$, it also has a Newton derivative.
\begin{corollary}\label{po.sond}
The stop operator
\begin{equation}\label{po.sond.1}
\stopr: X\times\real \to L^q(a,b) \,,
\quad 1\le q < \infty \,,
\end{equation}
has a Newton derivative given by
\begin{equation}\label{po.sond.2}
G^{S_r}(u,z_0) = \pi_1 - G^{P_r}(u,z_0)
\end{equation}
with elements
\begin{equation}\label{po.sond.3}
L^{S_r}(h,y) = h - L^{P_r}(h,y) \,.
\end{equation}
Here, $G^{P_r}$ and  $L^{P_r}$ have the form and properties described in
Theorem \ref{po.gnd}, and $\pi_1$ denotes the projection
$\pi_1(h,y) = h$.
\hfill$\Box$
\end{corollary}
\section{Bouligand derivative of the play and the stop}

We intend to prove that the play and the stop operator are Bouligand
differentiable from $X\times\real$ to $L^q$, $1\le q < \infty$.
It suffices to show that $\mathcal{P}_r,\mathcal{S}_r:X^\Delta\to L^q(a,b)$
are Bouligand differentiable for arbitrary partitions $\Delta$, as the
sets $X^\Delta\subset X\times\real$ are open and their union covers
$X\times\real$.

In the previous section we explained how, on $X^\Delta$, the play can
be represented as a finite composition of the positive
part $F_{pp}$, the cumulated maximum $F^I$ and continuous linear mappings.
By virtue of the chain rule, it therefore suffices to show that
$F_{pp}$ and $F^I$ are Bouligand differentiable, and that the
function spaces involved in the composition are fitting.

The positive part mapping $\beta:\real\to\real$,
$\beta(x) = \max\{x,0\}$ has the directional (in fact, Bouligand) derivative
\begin{equation}\label{bd.fpp.0}
\beta'(x;y) = \begin{cases} 0 \,,
& x < 0  \text{ or } x = 0 \,,\,y \le 0 \,,
\\ y \,, & x > 0 \text{ or } x = 0 \,,\,y > 0 \,.
\end{cases}
\end{equation}
\begin{lemma}\label{bd.fpp}
Let $I\subset [a,b]$ be a closed interval, $1\le q < \tilde{q} \le \infty$.
The mapping $F_{pp}: L^{\tilde{q}}(I) \to L^q(I)$,
$F_{pp}(u) = \max\{u,0\}$, is Bouligand differentiable, and
\begin{equation}\label{bd.fpp.1}
F_{pp}'(u;h)(t) = \beta'(u(t);h(t)) \,.
\end{equation}
\end{lemma}

\begin{proof}
See Examples 8.12 and 8.14 in \cite{IK}.
\end{proof}

It has already be proved in Proposition \ref{am.bounew} that the
cumulated maximum $F^I:X \to L^{\tilde{q}}(I)$ is Bouligand differentiable
for every $\tilde{q} < \infty$, and that
\begin{equation}\label{bd.fi.1}
(F^{I})'(u;h)(t) = \max_{s\in M^I(u,t)} h(s) \,.
\end{equation}
By the chain rule, the mapping $P^I_+: Z^I\cap (X\times\real)\to L^q(I)$,
\[
P_+^I(u,p) = p + F_{pp}(F^I(u-p-r))
\]
has the Bouligand derivative
\begin{equation}\label{bd.pi.1}
(P_+^I)'((u,p);(h,\eta))(t) = \eta +
\beta'(\max_{s\in I,s\le t} (u(s) - r - p);
\max_{s\in M^I(u,t)} (h(s) - \eta)) \,.
\end{equation}
An analogous formula holds for the Bouligand derivative of $P_-^I$.
Applying the chain rule to (\ref{po.wkd.0}), 
we obtain the Bouligand derivative of the play recursively as
\begin{equation}\label{bd.p.0}
\begin{split}
(w_{k}^\Delta)'((u,z_0);(h,y)) =
(P_k^I)'((u,w_{k-1}(u,z_0));(h,(w_{k-1}^\Delta)'((u,z_0);(h,y))))(t_k) \,,
\\
\playr'[[u;z_0];[h;y]](t) =
(P_k^I)'((u,w_{k-1}(u,z_0));(h,(w_{k-1}^\Delta)'((u,z_0);(h,y))))(t) \,.
\end{split}
\end{equation}
We also obtain the refined remainder estimate.
\begin{theorem}\label{bd.p}
The Bouligand derivative of the play operator $\mathcal{P}_r$ given in
(\ref{bd.p.0}) satisfies, for all $(u,z_0)\in X\times\real$, the remainder estimate
\begin{equation}\label{bd.p.1}
\begin{split}
&\| \playr[u+h;z_0+y] - \playr[u;z_0] - \playr'[[u;z_0];[h;y]]\|_{L^q(a,b)} 
\\ &\qquad \le \rho_{u,z_0}(\|h\|_\infty + |y|) \|(h,y)\|_{X\times\real}
\end{split}
\end{equation}
for all $h\in X$, $y\in\real$.
Here, $\rho_{(u,z_0)}(\delta)\downarrow 0$ as $\delta\downarrow 0$
and $\rho_{(u,z_0)}$ is uniformly bounded in $(u,z_0)$.
\end{theorem}
\begin{proof}
The proof is analogous to that for the Newton derivative, using
Proposition \ref{rr.boucomp} instead of Proposition \ref{rr.newcomp}.
\end{proof}

\section{The parametric play operator}

Instead of a single play operator acting on a function
$u = u(t)$, we now want to consider a family of play operators
acting on a function $u = u(x,t)$, where $x$ plays the role of
a parameter. This has been developed in \cite{Vis} in order to solve
boundary value problems for partial differential equations with
hysteresis. 
Here, we are concerned with parametrizing the Newton derivative of the play.

For a given measurable space $\Omega$ (that is, a set $\Omega$ equipped
with a sigma algebra), we want to define the \textbf{parametric play
operator} $\mathcal{P}_r^\Omega$ by
\begin{equation}\label{pp.1}
\mathcal{P}_r^\Omega[u;z_0](x,t) = \mathcal{P}_r[u(x,\cdot);z_0(x)](t)
\end{equation}
for functions $u:\Omega\times [a,b]\to\real$, $z_0:\Omega\to\real$.
The parametric play operator thus represents a parametric family of 
play operators.

For a given metric space $X$, equipped with the Borel sigma algebra,
let $\mathcal{M}(\Omega;X)$ denote the space of all measurable functions
from $\Omega$ to $X$.

\begin{lemma}\label{pp.ppdef}
Formula (\ref{pp.1}) defines an operator
\begin{equation}\label{pp.ppdef.1}
\mathcal{P}_r^\Omega: \mathcal{M}(\Omega;C[a,b])\times
\mathcal{M}(\Omega;\real)\to \mathcal{M}(\Omega;C[a,b]) \,.
\end{equation}
\end{lemma}
\begin{proof} The assignment
$x\mapsto (u(x,\cdot),z_0(x)) \mapsto \mathcal{P}_r[u(x,\cdot),z_0(x)]$
defines a mapping $\Omega\to C[a,b]\times\real \to C[a,b]$
which is measurable since $\mathcal{P}_r$ is continuous.
\end{proof}

We define the \textbf{parametric cumulated maximum} (that is,
the parametric family of cumulated maxima) $F^\Omega$
for functions $u:\Omega\to C[a,b])$ by
\begin{equation}\label{pp.am.0}
(F^\Omega u)(x) = F(u(x,\cdot)) \,,\quad x\in\Omega \,.
\end{equation}

\begin{lemma}\label{pp.am}
We have
\begin{equation}\label{pp.am.1}
\begin{split}
&F^\Omega: \mathcal{M}(\Omega;C[a,b])\to\mathcal{M}(\Omega;C[a,b]) \,,
\\
&F^\Omega: L^p(\Omega;C[a,b])\to L^p(\Omega;C[a,b]) \,,\quad 1\le p\le \infty \,.
\end{split}
\end{equation}
\end{lemma}

\begin{proof}
If $u:\Omega\to C[a,b]$ is measurable, the composition
$x\mapsto u(x,\cdot)\mapsto F(u(x,\cdot))$ defines a measurable
mapping since $F:C[a,b]\to C[a,b]$ is continuous. As
$\|(F^\Omega u)(x)\|_\infty \le \|u(x,\cdot)\|_\infty$
and because $u(x,\cdot) = v(x,\cdot)$ a.e. in $x$ implies that
$F^\Omega u = F^\Omega v$ a.e. in $x$, the second assertion in
(\ref{pp.am.1}) follows.
\end{proof}

The corresponding set-valued mappings $M^\Omega$ and $\Phi^\Omega$
are given by
\begin{equation}\label{pp.mphidef.1}
M^\Omega(u,t,x) = M(u(x,\cdot),t) \,,\quad
\Phi^\Omega(u,t,x) = \Phi(u(x,\cdot),t) \,.
\end{equation}
For any given function $u:\Omega\to C[a,b]$, these formulas define
set-valued mappings
\begin{equation}\label{pp.mphidef.2}
\begin{split}
&(x,t) \mapsto M(u(x,\cdot),t) = M^\Omega(u,t,x) \,,\quad
\Omega\times [a,b]\toto [a,b] \,,
\\
&(x,t) \mapsto \Phi(u(x,\cdot),t) = \Phi^\Omega(u,t,x) \,,\quad
\Omega\times [a,b]\toto C[a,b]^* \,.
\end{split}
\end{equation}

\begin{lemma}\label{pp.mphi}
Let $u\in \mathcal{M}(\Omega;C[a,b])$.
Then the mappings defined in (\ref{pp.mphidef.2}) are measurable.
\end{lemma}

\begin{proof}
The mappings arise as compositions
\begin{gather*}
(x,t) \mapsto (u(x,\cdot),t) \mapsto M(u(x,\cdot),t) \,,\quad
\Omega\times [a,b] \to C[a,b]\times [a,b] \toto [a,b] \,,
\\
(x,t) \mapsto (u(x,\cdot),t) \mapsto \Phi(u(x,\cdot),t) \,,\quad
\Omega\times [a,b] \to C[a,b]\times [a,b] \toto C[a,b]^* \,.
\end{gather*}
Due to Propositions \ref{mam.mtusc} and \ref{mam.phitusc}, the
assertion follows.
\end{proof}

In Proposition \ref{am.newder}, a Newton derivative $G$ of the
cumulated maximum $F$ has been constructed from measurable
selectors $\mu$ of $\Phi$.
Any such $\mu\in S_\Phi$ defines an element of $G(u(x,\cdot))$.
More precisely, given $u\in\mathcal{M}(\Omega;C[a,b])$
and $x\in\Omega$ we set
\begin{equation}\label{pp.amd.0}
[(L^\Omega(x))v](t) = \langle \mu(u(x,\cdot),t),v \rangle \,,\quad v\in C[a,b] \,.
\end{equation}
\begin{proposition}\label{pp.amd}
Let $\mu$ be a measurable selector of $\Phi$, let $u\in\mathcal{M}(\Omega;C[a,b])$.
Then (\ref{pp.amd.0}) defines a mapping
\begin{equation}\label{pp.amd.1}
L^\Omega:\Omega \to \mathcal{L}(C[a,b];L^\infty(a,b))
\end{equation}
with the property
\begin{equation}\label{pp.amd.2}
L^\Omega(x) \in G(u(x,\cdot)) \quad \text{for all $x\in\Omega$.}
\end{equation}
Let moreover $h\in\mathcal{M}(\Omega;C[a,b])$. Then
\begin{equation}\label{pp.amd.3}
(x,t) \mapsto [(L^\Omega(x))h(x,\cdot)](t) = \langle \mu(u(x,\cdot),t),h(x,\cdot) \rangle
\end{equation}
defines a measurable function from $\Omega\times [a,b]$ to $\real$.
\end{proposition}

\begin{proof}
Proposition \ref{am.newder} yields (\ref{pp.amd.1}) and (\ref{pp.amd.2}).
The composition $(x,t) \mapsto (u(x,\cdot),t) \mapsto \mu(u(x,\cdot),t)$
defines a measurable mapping from $\Omega\times [a,b]$ to $C[a,b]^*$,
since $\mu:C[a,b]\times [a,b] \to C[a,b]^*$ is measurable. As
the mapping $x\mapsto h(x,\cdot)$ is measurable and the mapping
$(\nu,v)\mapsto \langle\nu,v\rangle$ is continuous, (\ref{pp.amd.3})
defines a measurable function.
\end{proof}

We define
\begin{equation}\label{pp.amd.8}
\begin{split}
G^\Omega: \mathcal{M}(\Omega;C[a,b]) \toto
\text{Map}(\Omega;\mathcal{L}(C[a,b];L^\infty(a,b))) \qquad
\\
G^\Omega(u) = \{L^\Omega: \text{$L^\Omega$ satisfies (\ref{pp.amd.0}) and (\ref{pp.amd.1})
for some $\mu\in S_\Phi$}\} \,,
\end{split}
\end{equation}
a parametric family of Newton derivatives of the parametric family of cumulated maxima
$F^\Omega$.
It is \textbf{not} a Newton derivative of $F^\Omega$.
(Here, $\text{Map}(A;B)$ stands for the set of all mappings from a set $A$ to a set $B$.)

For the parametric play $\mathcal{P}_r^\Omega$ we proceed in the same manner.
According to Theorem \ref{po.gnd}, the Newton derivative $G^{P_r}$ of $\mathcal{P}_r$
constructed there has, when evaluated at $(u,z_0)$, elements of the form
\begin{equation}\label{pp.pnd.01}
L^{P_r}(h,y)(t) = \langle \mu^{P_r}(u,z_0,t), (h,y)\rangle \,,\quad t\in (a,b)\,,
\end{equation}
for some $\mu^{P_r}$ as given in (\ref{po.gnd.5}). We define
\begin{equation}\label{pp.pnd.02}
\begin{split}
L_r^\Omega:\Omega \to \mathcal{L}(C[a,b]\times\real;L^\infty(a,b)) \qquad \quad
\\
[(L_r^\Omega(x))(v,y_0)](t) = \langle \mu^{P_r}(u(x,\cdot),z_0(x),t),(v,y_0) \rangle
\end{split}
\end{equation}
for $(v,y_0)\in C[a,b]\times\real$.
\begin{proposition}\label{pp.pnd}
Let $\mu^{P_r}$ be as given in (\ref{po.gnd.5}), let $u\in\mathcal{M}(\Omega;C[a,b])$
and $z_0\in \mathcal{M}(\Omega;\real)$.
Then $L_r^\Omega$ as given in (\ref{pp.pnd.02}) satisfies
\begin{equation}\label{pp.pnd.1}
L_r^\Omega(x) \in G^{P_r}(u(x,\cdot),z_0(x)) \quad \text{for all $x\in\Omega$.}
\end{equation}
Let moreover $h\in\mathcal{M}(\Omega;C[a,b])$, $y\in \mathcal{M}(\Omega;\real)$. Then
\begin{equation}\label{pp.pnd.2}
(x,t) \mapsto [(L_r^\Omega(x))(h(x,\cdot),y(x)](t)
= \langle \mu^{P_r}(u(x,\cdot),z_0(x),t),(h(x,\cdot),y(x)) \rangle
\end{equation}
defines a measurable function from $\Omega\times [a,b]$ to $\real$.
\end{proposition}

\begin{proof}
Analogous to that of Proposition \ref{pp.amd}.
\end{proof}

We may define $G^\Omega_r(u,z_0)$ as the set of all such mappings $L_r^\Omega$
and view $G^\Omega_r$ as a parametric Newton derivative of the
parametric play $\mathcal{P}_r^\Omega$.

\section{Appendix: set-valued mappings}

In this section, we recall some standard results from set-valued analysis,
given e.g. in \cite{PK}, and derive some consequences needed in this paper.

Let $\Psi:X\toto Y$. We generally assume that $\Psi(u)\neq\emptyset$ for
every $u\in X$.

\begin{definition}\label{svm.uscdef}
Let $X,Y$ be Hausdorff topological spaces, let $\Psi:X\toto Y$.
We say that $\Psi$ is \textbf{upper semicontinuous}
(or \textbf{usc} for short), if
\begin{equation}\label{mam.uscdef.1}
\Psi^{-1}(A) := \{u: u\in X,\,\Psi(u)\cap A \neq \emptyset\}
\end{equation}
is closed for every closed subset $A$ of $Y$.
We say that $\Psi$ is \textbf{measurable} if $\Psi^{-1}(V)$ is
measurable for all open $V\subset Y$.
A mapping $\psi:X\to Y$ is called a \textbf{measurable selector} of
$\Psi$ if $\psi$ is measurable and $\psi(u)\in \Psi(u)$ for every $u\in X$.
\end{definition}

\begin{lemma}\label{svm.uscequiv}
Let $X,Y$ be Hausdorff topological spaces.
A mapping $\Psi:X\toto Y$ is usc if and only if for every $u\in X$ and
every open set $V$ with $\Psi(u)\subset V\subset Y$ there exists an open set
$U\subset X$ with $u\in U$ and $\Psi(U)\subset V$.
\end{lemma}

\begin{proof}
See Proposition 6.1.3 in \cite{PK}.
\end{proof}

Obviously, a single-valued mapping is continuous in the usual sense if
and only if it is usc in the sense above.

The following two lemmas are immediate consequences of Lemma \ref{svm.uscequiv}.

\begin{lemma}\label{svm.uscrest}
Let $X,Y$ be Hausdorff topological spaces, $X_0\subset X$.
Let $\Psi:X\toto Y$ be usc. Then $\Psi|X_0:X_0\toto Y$ is usc.
\hfill$\Box$
\end{lemma}

\begin{lemma}\label{svm.uscext}
Let $X,Y$ be Hausdorff topological spaces, $X = X_1 \cup X_2$
with $X_1,X_2$ open.
Let $\Psi_j:X_j\toto Y$ be usc for $j=1,2$ such that
$\Psi_1|(X_1\cap X_2) = \Psi_2|(X_1\cap X_2)$.
Then $\Psi:X\toto Y$ defined by $\Psi(u) = \Psi_j(u)$ if $u\in X_j$
is usc.
\hfill$\Box$
\end{lemma}

The composition $\Psi_2\circ\Psi_1$ of two set-valued mappings
$\Psi_1:X\toto Y$ and $\Psi_2:Y\toto Z$ is defined as
\begin{equation}\label{comp.1}
(\Psi_2\circ\Psi_1)(u) = \bigcup_{v\in\Psi_1(u)} \Psi_2(v) \,.
\end{equation}

\begin{lemma}\label{svm.comp}
Let $X,Y,Z$ be Hausdorff topological spaces, let $\Psi_1:X\toto Y$
and $\Psi_2:Y\toto Z$ be usc. Then $\Psi_2\circ \Psi_1$ is usc.
\end{lemma}

\begin{proof}
This is again straightforward, using Lemma \ref{svm.uscequiv}.
See Proposition 2.56 in \cite{HP}.
\end{proof}

We will use Lemma \ref{svm.comp} mainly for the special cases
$\Psi\circ f$ and $f \circ \Psi$ where $\Psi$ is usc and
$f$ is single-valued and continuous. 

\begin{lemma}\label{svm.complin}
Let $X,Y,Z$ be normed spaces, $U\subset X$ and $V\subset Y$ open.
Let $f:U\to V$, $g_1:U\to L(X,Y)$ and $g_2:V\to L(Y,Z)$ be
measurable. Then $g:U\to L(X,Z)$ defined by
$g(u) = g_2(f(u))\circ g_1(u)$ is measurable.
\end{lemma}

\begin{proof}
The composition is a continuous mapping from $L(X,Y)\times L(Y,Z)$
to $L(X,Z)$.
\end{proof}

%

\begin{proposition}\label{svm.grc}
Let $X,Y$ be Hausdorff topological spaces, let $\Psi:X\toto Y$.
We assume that $\Psi$ has compact values, that is,
$\Psi(u)$ is compact for all $u\in X$.
\hfill\\ (i)
If $\Psi$ is usc, then the graph of $\Psi$,
\begin{equation}\label{svm.grc.1}
{\rm Gr}\,\Psi = \{(u,v): u\in X,\, v\in \Psi(u)\}
\end{equation}
is closed in $X\times Y$.
\hfill\\ (ii)
If the graph of $\Psi$ is closed in $X\times Y$ and if $\Psi(X)$
is relatively compact in $Y$, then $\Psi$ is usc.
\end{proposition}

\begin{proof}
See Proposition 6.1.8, Remark 6.1.9 and Proposition 6.1.10 in \cite{PK}.
\end{proof}

Above we consider compositions of the form
\begin{equation}\label{svm.psicomp.1}
F(u,p,t) = F_2(u,p,F_1(u,p,t),t)
\end{equation}
for mappings $F_1:U\times\real\times I\to\real$ and
$F_2:U\times\real\times\real\times I\to\real$, where $U\subset C(J)$.
Here we are concerned with the upper semicontinuity of a
corresponding composition of mappings $\Psi_1,\Psi_2$
arising in the construction of Newton derivatives.
\begin{proposition}\label{svm.psicomp}
Let $I,J\subset\real$ be compact intervals, $U\subset C(J)$ open.
Let $F_1:U\times\real\times I\to\real$ be continuous.
Let $\Psi_1: U\times\real\times I \toto (C(J)\times\real)^*$ and
$\Psi_2: U\times\real\times\real\times I \toto (C(J)\times\real\times\real)^*$
be usc, with $w^*$-compact values, and locally bounded.
Let $\Psi: U\times\real\times I \toto (C(J)\times\real)^*$
be defined by
\begin{equation}\label{svm.psicomp.2}
\Psi(u,p,t) = \{ L_2\circ (\text{id},L_1): L_1\in \Psi_1(u,p,t),\,
L_2\in \Psi_2(u,p,F_1(u,p,t),t)\} \,,
\end{equation}
where id denotes the identity on $C(J)\times\real$.
Then $\Psi$ is usc, has $w^*$-compact values and is locally bounded.
\end{proposition}

\begin{proof}
As $\Psi_1$ and $\Psi_2$ are locally bounded, we see from (\ref{svm.psicomp.2})
and the continuity of $F_1$ that $\Psi$ is locally bounded.
Next, let $(L_1^n)$ and $(L_2^n)$ be arbitrary sequences in $(C(J)\times\real)^*$
and $(C(J)\times\real\times\real)^*$ respectively.
We claim that
\begin{equation}\label{svm.psicomp.3}
L_1^n \wsto L_1 \,,\, L_2^n \wsto L_2 \quad \Rightarrow \quad
L_2^n \circ (\text{id},L_1^n) \wsto L_2 \circ (\text{id},L_1)  \,.
\end{equation}
Indeed, for any $(h,q)\in C(J)\times\real$ we have
\[
\langle L_2^n \,,\, (h,q,L_1^n(h,q))\rangle \to
\langle L_2 \,,\, (h,q,L_1(h,q))\rangle \,.
\]
To prove that $\Psi$ has $w^*$-compact values, let
$L^n = L_2^n \circ (\text{id},L_1^n)$ be a sequence in $\Psi(u,p,t)$.
By assumption, passing to suitable subsequences we have
$L_1^n \wsto L_1\in\Psi_1(u,p,t)$ and
$L_2^n \wsto L_2\in\Psi_2(u,p,F_1(u,p,t),t)$. By (\ref{svm.psicomp.3}),
$L^n \wsto L_2 \circ (\text{id},L_1) \in\Psi(u,p,t)$.
It remains to prove that $\Psi$ is usc.
Let $A\subset (C(J)\times\real)^*$ be $w^*$-closed; it suffices to
show that $\Psi^{-1}(A)$ is closed.
Let $(u_n,p_n,t_n)\in \Psi^{-1}(A)$ and $(u_n,p_n,t_n)\to (u,p,t)$.
Let $L^n\in\Psi(u_n,p_n,t_n)\cap A$.
We have $L^n = L_2^n \circ (\text{id},L_1^n)$ for some
$L_1^n\in \Psi_1(u_n,p_n,t_n)$ and
$L_2^n\in \Psi_2(u_n,p_n,F_1(u_n,p_n,t_n),t_n)$.
Since $\Psi_1$ and $\Psi_2$ are locally bounded, passing to a subsequence
we get $L_1^n\wsto L_1$, $L_2^n\wsto L_2$.
As the graphs of $\Psi_1$ and $\Psi_2$ are $w^*$-closed by Proposition
\ref{svm.grc},
$L_1\in\Psi_1(u,p,t)$ and $L_2\in\Psi_2(u,p,F_1(u,p,t),t)$.
By (\ref{svm.psicomp.3}),
$L^n \wsto L_2 \circ (\text{id},L_1) =: L \in\Psi(u,p,t)$. As $A$
is $w^*$-closed, it follows that $L\in A$.
Thus, $\Psi^{-1}(A)$ is closed.
\end{proof}

\textbf{Acknowledgements.} 
The author thanks Michael Ulbrich in particular for pointing out the 
line of argument used in the proof of Propositions \ref{am.bounew}
and \ref{am.newder}, and him as well as Constantin Christof, Michael Hinterm\"{u}ller,
Pavel Krej\v{c}\'{\i}, Karl Kunisch and Gerd Wachsmuth for valuable discussions.

\end{document}